\documentclass{amsart}

\usepackage{amssymb}
\usepackage{enumitem}

\newtheorem{theorem}{Theorem}[section]
\newtheorem{lemma}[theorem]{Lemma}
\newtheorem{proposition}[theorem]{Proposition}
\newtheorem{corollary}[theorem]{Corollary}

\theoremstyle{definition}

\newtheorem{remark}[theorem]{Remark}

\numberwithin{equation}{section}

\def\endprf{\hspace*{\fill}$\square$}

\DeclareMathOperator{\spn}{span}

\newcommand{\tightoverset}[2]{%
  \mathop{#2}\limits^{\vbox to -.6ex{\kern-.6ex\hbox{$#1$}\vss}}}
\renewcommand{\vec}{\tightoverset{\scriptstyle{\rightharpoonup}}}

\renewcommand{\epsilon}{\varepsilon}

\def\arr{\rightarrow}

\def\implies{\Longrightarrow}
\def\iff{\Leftrightarrow}
\def\restr{\upharpoonright}
\def\nin{\not \in}

\begin{document}

\title{Abelian group actions and hypersmooth equivalence relations}

\author{Michael R. Cotton}

\address{Department of Mathematics, University of North Texas, Denton, Texas 76203}

\email{Michael.Cotton@unt.edu}

\thanks{This research was conducted for the completion of a doctoral dissertation written under the direction of Professor Su Gao at the University of North Texas and was partially supported by NSF grant DMS-0943870. The author would also like to thank Stephen Jackson for many helpful discussions and comments as well as the reviewer of this manuscript for several valuable corrections and improvements.}



\begin{abstract}
We show that a Borel action of a standard Borel group which is isomorphic to a sum of a countable abelian group with a countable sum of real lines and circles induces an orbit equivalence relation which is hypersmooth, i.e., Borel reducible to eventual agreement on sequences of reals, and it follows from this result along with the structure theory for locally compact abelian groups that Borel actions of Polish LCA groups induce orbit equivalence relations which are essentially hyperfinite, extending a result of Gao and Jackson and answering a question of Ding and Gao.
\end{abstract}

\maketitle

\section{Introduction}
To compare classification problems we need to be able to decide how complicated the corresponding equivalence relations are, and the comparison of complexity that we focus on in this paper is \emph{Borel reducibility} which essentially says there is a reasonably definable way of deciding one equivalence given that we can decide another. It is important in this context that we require the reduction to be definable in some way such as Borel. If all we required was that there exists an injection from one quotient space into the other then the only thing the reduction would compare is the number of equivalence classes, which has very little to do with the complexity of the relations.

Many important equivalence relations are actually \emph{orbit equivalence relations}, meaning the equivalence classes may be realized as the orbits of some reasonably definable group action. For example Vitali equivalence, $r\sim t \Leftrightarrow r-t\in\mathbb{Q}$, is induced by the orbits of $\mathbb{Q}$ acting on $\mathbb{R}$ by translation. For a more sophisticated
example, the finitely generated groups may be coded as a Polish space on which the isomorphism relation for finitely generated groups is induced by the image action of the group of those automorphisms of the free group on countably many generators which only move finitely many of the generators (Champetier, \cite{champetier}).

One class of equivalence relations of particular interest are the hyperfinite relations whose equivalence classes are precisely the orbits of a single Borel automorphism of the space or, equivalently, the orbits of an action of the integers. It has been known for some time that Borel $\mathbb{Z}^n$-actions are still of this type, hence the orbits of a Borel $\mathbb{Z}^n$-action are realizable as the orbits of a single Borel automorphism. But it was not until \cite{gao_jackson} that Gao and Jackson showed that the orbit equivalence induced by a Borel action of \emph{any} countable abelian group is hyperfinite. So in terms of reasonably definable orbit equivalence, no countable abelian group can do worse than $\mathbb{Z}$. Improving on this can go in multiple directions. One is that the abelian property of the group is stronger than necessary. It was shown by Schneider and Seward in \cite{schneider_seward} that it is sufficient for the countable group to be locally nilpotent, and this has been carried further in \cite{cjmst} where Conley, Jackson, Marks, Seward, and Tucker-Drob achieve the result for polycyclic groups and the free actions of a large class of solvable groups. It has long been conjectured that any countable amenable group will only induce such relations, but this currently remains unknown.

Another direction is to do away with the requirement for the group to be countable. In this case, we would aim to show that the induced orbit equivalence relation is \emph{essentially hyperfinite}, meaning that it is Borel reducible to a hyperfinite one. It cannot be exactly hyperfinite since the equivalence classes are no longer countable, but for our notion of complexity the reducibility result is enough. For example, the relation on $\mathbb{R}$ in which all points are equivalent to each other is a very simple relation with only a single class, but it is of course not induced by any $\mathbb{Z}$-action on $\mathbb{R}$. However, we would say that it is \emph{essentially} hyperfinite since it would reduce to any $\mathbb{Z}$-orbit equivalence by way of a constant function. Ding and Gao showed in \cite{ding_gao} that any essentially countable Borel equivalence relation which is Borel reducible to an orbit equivlance relation induced by a non-archimedean abelian Polish group must be essentially hyperfinite. As a corollary, this means in particular that locally compact non-archimedean abelian Polish groups induce only essentially hyperfinite Borel equivalence relations.

So we arrive at the motivation for this paper. Our desire is to explore the descriptive complexity of those orbit equivalence relations which are induced by an action of an abelian standard Borel group which we no longer require to be countable and also wish to avoid the need to be non-archimedean as that case is now well-understood. And we have made some progress in this direction.

First, we are able to conclude that a larger subcategory of abelian standard Borel groups than the locally compact ones will only induce relations which are a level above the hyperfinite ones, but only just. While the essentially hyperfinite equivalence relations are those which are Borel reducible to eventual agreement of sequences of naturals, the \emph{hypersmooth} equivalence relations are those which are reducible to eventual agreement of sequences of reals. We will give more basic definitions later, but note that the hypersmooth are a minimal step above the hyperfinite in the sense that there are no strictly intermediate equivalence relations under Borel reduction.

We have the following result.

\begin{theorem}\label{sumtheorem} A Borel action on a standard Borel space of a group which is (Borel) isomorphic to the sum of a countable abelian group with a countable sum of copies of $\mathbb{R}$ and $\mathbb{T}$ induces a Borel orbit equivalence relation which is hypersmooth.
\end{theorem}

Then, it follows that we can use the known structure theory for Hausdorff locally compact abelian (i.e., LCA)  groups to reduce those relations induced by any Polish LCA group to the hyperfinite ones, providing a nice extension of Gao and Jackson's result for the countable abelian groups (although in the case of free actions our method does not provide a continuous reduction) and eliminating the need for the non-archimedean property in the LCA case.

\begin{theorem}\label{LCAtheorem} A Borel action of a second countable LCA group on a standard Borel space induces a Borel orbit equivalence relation which is essentially hyperfinite.
\end{theorem}

In Section \ref{reductionsection} we show that all of the orbit equivalence relations from Theorems \ref{sumtheorem} and \ref{LCAtheorem} reduce to ones induced by a Borel action of a countable sum of copies of $\mathbb{R}$. In Section \ref{reductiontofreesection} we show that these orbit equivalence relations which are induced by a countable sum of copies of $\mathbb{R}$ are countable disjoint unions of equivalence relations which are induced by free Borel actions of countable sums of copies of $\mathbb{R}$ and $\mathbb{T}$. Finally, in Section \ref{freesection}, we assume familiarity with Gao and Jackson's argument from \cite{gao_jackson} for a free action of a countable sum of copies of $\mathbb{Z}$ and show how it is modified to conclude the relations induced by the free actions of countable sums of copies of $\mathbb{R}$ and $\mathbb{T}$ are hypersmooth. Since a countable disjoint union of hypersmooth equivalence relations is hypersmooth, locally compact Polish groups induce essentially countable equivalence relations, and essentially countable equivalence relations which are hypersmooth must be essentially hyperfinite, we will have proven Theorems \ref{sumtheorem} and \ref{LCAtheorem}.

\section{Preliminaries}
\subsection{Standard Borel Spaces}
A \emph{Polish space} is a separable, completely metrizable topological space.

The collection of \emph{Borel sets} of a topological space is the $\sigma$-algebra generated by the open sets. In some settings, the collection of Borel sets without distinguishing the open sets is more important than any particular topology involved. So rather than a topology on a space, we will often be interested only in a \emph{Borel structure} on the space. We say that a function is a \emph{Borel function} (or Borel measurable function) if the preimages of Borel sets are Borel.

The proof of the following theorem is surprisingly nontrivial:  (See $\S 15A$ in \cite{kechris}.)

\begin{theorem}\emph{(Luzin--Suslin)}
If $X$ and $Y$ are Polish spaces and $f:X\arr Y$ is a Borel injection, then for any Borel set $A$ in $X$, $f(A)$ is Borel in $Y$.
\end{theorem}

Thus, one-to-one Borel functions have Borel inverses. So if we suppose there exists a Borel bijection $f$ from $X$ onto $Y$, it follows that $A$ is Borel in $X$ iff $f(A)$ is Borel in $Y$. And we can say that $X$ and $Y$ are \emph{Borel isomorphic}.
 
When working in the generality of Borel structures, we take advantage of the \emph{Borel isomorphism theorem}:
(Proofs can be found in $\S 15B$ of \cite{kechris} or $\S 1.3$ of \cite{gao}.)
\begin{theorem}\label{borelisomorphism}
Any two uncountable Polish spaces are Borel isomorphic.
\end{theorem}

By the Borel isomorphism theorem, the Borel structures on uncountable Polish spaces are all essentially the same. In particular, any two Polish topologies on an uncountable set must generate the same Borel sets. So a set which is Borel in \emph{some} Polish topology must be Borel in \emph{all} Polish topologies, and so we may talk about \emph{the} Borel structure on such a space.

A \emph{standard Borel space} is a measurable space $(X,\mathcal{A})$ where there exists a Polish topology on $X$ so that $\mathcal{A}$ is the $\sigma$-algebra of Borel sets generated by the topology.

The first step in attacking many classification problems is to parametrize the collection of interesting objects as a standard Borel space so that we can apply the tools of descriptive set theory. Often each object can be embedded into a Polish space as a subset. So once we have these embeddings we need to know if we can build a standard Borel space out of those subsets, and one way to do that is to use the Effros structure.

Let $F(X)$ denote the collection of closed subsets of a space $X$. The \emph{Effros Borel structure} on $F(X)$ is the $\sigma$-algebra generated by the sets of the form $\{F\in F(X):F\cap U\neq\emptyset\}$ where $U\subseteq X$ is open. (For the following, see $\S 12C$ in \cite{kechris}.)
\begin{theorem}
If $X$ is a Polish space, the space $F(X)$ with the Effros structure is a standard Borel space.
\end{theorem}

Finally, we need a good tool for showing that a subspace of the Effros space is Borel or often for just helping us definably choose points in a construction. (See $\S 12C$ in \cite{kechris}.)
\begin{theorem}\label{selection}\emph{(Kuratowski--Ryll-Nardzewski)}
For a Polish space $X$, there is a sequence of Borel functions $f_{n}:F(X)\setminus\{\emptyset\}\arr X$ such that $\{f_{n}(F)\}$ is dense in $F$ for each $F\in F(X)\setminus\{\emptyset\}$.
\end{theorem}

\subsection{Some Classifications of Equivalence Relations}

For two equivalence relations $E$ and $F$ on standard Borel spaces $X$ and $Y$, resp., we say that $E$ is \emph{Borel reducible} to $F$, denoted $E\leq_{B}F$, if there exists a Borel function $f:X\arr Y$ such that $x_{1}\, E\, x_{2}\iff f(x_{1})\, F\, f(x_{2})$. Equivalently, $E\leq_{B}F$ if there is an injection with a Borel lifting from the equivalence classes $X/E$ into the equivalence classes $Y/F$. We say $E$ is \emph{bireducible} to $F$ if both $E\leq_{B}F$ and $F\leq_{B}E$.

We call an equivalence relation \emph{smooth} if it is Borel reducible to $=$ on the reals, or equivalently by the Borel isomorphism theorem, if it is reducible to the identity relation on $X$ for any uncountable Polish space $X$. The smooth equivalence relations are also called \emph{concretely classifiable} since in this case the Borel reduction is providing a reasonably definable procedure for computing a number, or concrete invariant, which decides the classification of the object. A useful fact for us is the following: (See $\S 5.4$ in \cite{gao}.)\begin{theorem}\label{closed}
Closed equivalence relations are smooth.
\end{theorem}

We call a Borel equivalence relation a \emph{countable} equivalence relation if each equivalence class is countable. Similarly, we call an equivalence relation a finite relation if each equivalence class is finite. And then a countable Borel equivalence relation $E$ is called \emph{hyperfinite} if there is an increasing sequence $F_{0}\subseteq F_{1}\subseteq F_{2}\subseteq\ldots$ of finite Borel equivalence relations with $E=\bigcup_{n}F_{n}$, i.e., $x E y\Leftrightarrow \exists n (x F_{n}y)$. However, the term hyperfinite is only used to describe countable equivalence relations. In our wider context, an equivalence relation is called \emph{essentially countable} if it is Borel reducible to a countable Borel equivalence relation, and an equivalence relation is called \emph{essentially hyperfinite} if it is reducible to a hyperfinite Borel equivalence relation.

We call an equivalence relation \emph{hypersmooth} if there is an increasing sequence $F_{0}\subseteq F_{1}\subseteq F_{2}\subseteq\ldots$ of smooth Borel equivalence relations with $E=\bigcup_{n}F_{n}$, i.e., $x E y\Leftrightarrow \exists n (x F_{n}y)$.

It is useful to be able to show equivalence relations belong to a certain Borel reducibility class by having benchmark relations to reduce them to, and we have these for the classes we consider here. Define $E_0$ to be the equivalence relation on $\mathbb{N}^\mathbb{N}$ defined by $\vec{x} E_0 \vec{y}\Leftrightarrow\exists n\forall m\geq n\,(x_m =y_m)$. And define $E_1$ to be the equivalence relation on $\mathbb{R}^\mathbb{N}$ defined by $\vec{x} E_1 \vec{y}\Leftrightarrow\exists n\forall m\geq n\,(x_m=y_m)$. The following fact (See \cite{dougherty_jackson_kechris} and \cite{kechris_louveau}) establishes that $E_0$ and $E_1$ indeed are benchmarks for the essentially hyperfinite and hypersmooth relations in the same way that $=_{\mathbb{R}}$ is a benchmark for the smooth.

\begin{proposition} A Borel equivalence relation $E$ is hypersmooth iff $E\leq_{B}E_1$, and $E$ is essentially hyperfinite iff $E\leq_{B}E_0$.
\end{proposition}

Another basic fact that we will take advantage of is the following:
\begin{lemma}\label{hypersmoothunion} A countable disjoint union of hypersmooth equivalence relations is hypersmooth.
\end{lemma}
\begin{proof} Suppose $X=\bigcup_{n}X_n$ with each $X_{n}$ Borel and $i\neq j\Rightarrow X_{i}\cap X_{j}=\emptyset$, and suppose each $F_{n}$ is a hypersmooth equivalence relation on $X_{n}$. Then let each $f_{n}:X_{n}\rightarrow\mathbb{R}^\mathbb{N}$ be a Borel reduction of $F_{n}$ to $E_1$, i.e., $x F_{n} y\Leftrightarrow \vec{f_{n}(x)}E_1\vec{f_{n}(y)}$. Letting $\phi:\mathbb{R}\times\mathbb{N}\rightarrow\mathbb{R}$ be a Borel isomorphism and defining $n_{x}=m\Leftrightarrow x\in X_{m}$, it follows that $x\mapsto\langle\phi(\langle f_{n_x}(x)\rangle_{k}, n_{x})\rangle_{k\in\mathbb{N}}$ is a Borel reduction of the union equivalence $E=\bigcup_{n}F_{n}$ to $E_1$.
\end{proof}

Kechris and Louveau established the following dichotomy between the essentially hyperfinite and hypersmooth relations.
\begin{theorem}\label{hypersmoothdichotomy}\emph{(Kechris--Louveau, \cite{kechris_louveau})} Let $X$ be a standard Borel space and $E$ a hypersmooth equivalence relation on $X$. Then either $E$ is essentially hyperfinite or $E$ is bireducible to $E_1$.
\end{theorem}

Another important property of $E_1$ is that it cannot be reduced to any countable Borel equivalence relation.
\begin{theorem}\label{E1notcountable}\emph{(\cite{kechris_louveau} or \S 5 of \cite{kechris})} If $F$ is a countable Borel equivalence relation, then $E_1\not\leq_{B} F$.
\end{theorem}

Combining Theorems \ref{hypersmoothdichotomy} and \ref{E1notcountable}, we have the following.
\begin{corollary}\label{hypersmoothtohyperfinite} If $E$ is both hypersmooth and essentially countable, then $E$ is essentially hyperfinite.\end{corollary}

\subsection{Topological Groups and Orbit Equivalence Relations}
A \emph{topological group} is a group $G$ together with a topology on $G$ such that the group operation and inverse function are continuous, and a \emph{Polish group} is a topological group whose topology is Polish. More generally, a \emph{standard Borel group} is a group $G$ together with a standard Borel structure such that the group operations are Borel. So all Polish groups are standard Borel groups, and any Borel subgroup of a Polish group is a standard Borel group. However, not all standard Borel groups are Polish groups since there may not exist a Polish topology which induces the Borel structure while at the same time making the operations continuous.

By \emph{orbit equivalence relation}\footnote{One should note that the term ``orbit equivalence relation'' more commonly refers to an action of a Polish group in the literature, and in fact any analytic equivalence relation is the orbit equivalence relation of \emph{some} standard Borel group (but not necessarily of a Polish group).} we mean an equivalence relation on a standard Borel space where the equivalence classes are precisely the orbits of a Borel action on the space by a standard Borel group, i.e., an orbit equivalence relation $E$ on $X$ is such that there exists such a group $G$ and a Borel action $(g,x)\mapsto g\cdot x$ where $x\, E\, y\iff \exists g\in G\, \text{ s.t. } y=g\cdot x\,$. In the case when $E$ is an orbit equivalence relation on $X$ given by a Borel action of some group $G$, we will often summarize the situation by writing $E$ instead as $E^X_G$.

In the case where the acting group is Polish, Becker and Kechris provide the following powerful theorem.
\begin{theorem}\label{continuousaction}\emph{(Becker--Kechris, \cite{becker_kechris})}
Given a Borel action of a Polish group $G$ on a standard Borel space $X$, there exists a Polish topology on $X$ so that the action is continuous.
\end{theorem}

\begin{corollary}\label{closedstabilizers}Given a Borel action of a Polish group $G$ on a standard Borel space $X$, each stabilizer $G_{x}=\{g\in G:g\cdot x=x\}$ is a closed subgroup of $G$.
\end{corollary}
While Corollary \ref{closedstabilizers} follows easily from Theorem \ref{continuousaction}, it was originally proven directly by Miller in \cite{miller}.

They also show that the map from a point to its stabilizer is a Borel map into the Effros space of the acting group, and this fact will be important for some of our computations.
\begin{theorem}\label{boreltostabilizer}\emph{(Becker--Kechris, \cite{becker_kechris})}
Given a Borel action of a Polish group $G$ on a standard Borel space $X$, the map $x\mapsto G_x$ from $X$ into $F(G)$ is Borel if and only if $E^X_G$ is Borel.
\end{theorem}

The following definitions and theorems are useful. For proofs of these, we refer the reader to \cite{lacunary} or $\S 5.4$ of \cite{gao}.

We call an equivalence relation $E$ on a standard Borel space $X$ \emph{idealistic} if there is an assignment to each $E$-equivalence class $C$ of a nontrivial $\sigma$-ideal $I_C$ on $C$ such that for every Borel set $A\subseteq X^2$ the set $A_I$ where $x\in A_I \iff \{y\in [x]:(x,y)\in A\}\in I_{[x]}$ is Borel. By a \emph{Borel selector} for $E$, we mean a Borel function $s:X\rightarrow X$ such that for each $x,y\in X$, $s(x)E x$ and $xE y \Rightarrow s(x)=s(y)$

\begin{theorem}\label{idealistic}\emph{(Kechris, \cite{lacunary})} Let $E$ be an equivalence relation on a standard Borel space $X$. Then $E$ has a Borel selector iff $E$ is smooth and idealistic.\end{theorem}

In the case of of an orbit equivalence relation induced by an action of a Polish group, we may let $S\in I_{C}\Leftrightarrow\{g\in G:g\cdot x\in S\}$ is meager in $G$, and we get the following.
\begin{lemma}\label{borelselector}Let $G$ be a Polish group acting in a Borel manner on a Polish space $X$. Then $E_G^X$ is idealistic.\end{lemma}

In our constructions for Section \ref{freesection}, we will begin with the following ``countable sections'' constructed by Kechris.
\begin{theorem}\label{lacunary}\emph{(Kechris, \cite{lacunary})}
Suppose $G$ is a locally compact Polish group acting continuously on a Polish space $X$. Then there is a Borel set $Y\subseteq X$ which contains at least one and at most countably many points of each $G$-orbit (i.e., each $E_G^X$-equivalence class). Moreover, given any compact symmetric neighborhood $K$ of the identity in $G$, we may construct $Y$ so that $\forall y\in Y$, $(K\cdot y)\cap Y=\{y\}$. \emph{(Here, we will say that $Y$ is \emph{$K$-discrete})}.
\end{theorem}

Applying the Luzin--Novikov uniformization theorem to $P\subset X\times Y$ defined by $(x,y)\in P\Leftrightarrow x E^X_G y$, we get a Borel function which reduces $E^X_G$ to $E^X_{G}\upharpoonright Y$ and so we have the following corollary which, in light of Corollary \ref{hypersmoothtohyperfinite}, tells us that whenever an orbit equivalence relation which is induced by an action of a locally compact Polish group is hypersmooth then it is essentially hyperfinite.
\begin{corollary}\label{essentiallycountable}\emph{(Kechris, \cite{lacunary})}
If $G$ is a locally compact Polish group acting in a Borel manner on a standard Borel space $X$, then the induced orbit equivalence relation $E_G^X$ is essentially countable.
\end{corollary}

\section{Actions of LCA groups and reduction to $\mathbb{R}^{<\omega}$}\label{reductionsection}

We denote the groups $\bigoplus_{\omega}\mathbb{R}$, $\bigoplus_{\omega}\mathbb{T}$, and $\bigoplus_{\omega}\mathbb{Z}$ as $\mathbb{R}^{<\omega}$, $\mathbb{T}^{<\omega}$, and $\mathbb{Z}^{<\omega}$, respectively, where $\mathbb{T}\cong\mathbb{R}/\mathbb{Z}$ is viewed as $[0,1)$ with addition modulo $1$ as the operation.

\subsection{Reducing the LCA-actions to $\mathbb{R}^{n}\times\mathbb{Z}^{<\omega}$-actions} 
We make use of the Principal Structure Theorem for locally compact abelian groups. Following the convention of \cite{morris}, we say \emph{LCA group} to mean locally compact Hausdorff abelian topological group. The following version\footnote{As in \cite{morris}, the Principal Structure Theorem for LCA groups is often given as the seemingly weaker: ``Every LCA group has an open subgroup topologically isomorphic to $\mathbb{R}^{n}\times K$, for some compact group $K$ and non-negative integer $n$.'' However, this can be shown to be equivalent to the version used here.} of the structure theorem for LCA groups is carefully proven in \cite{deitmar_echterhoff}.

\begin{theorem}\label{structurethm}\emph{(Principal Structure Theorem)}  Every LCA group is topologically isomorphic to $\mathbb{R}^{n}\times H$ for some non-negative integer $n$ and where $H$ is some Hausdorff abelian group which has a compact open (in $H$) subgroup $K$.\end{theorem}

Now, we show that in our context of Borel reduction of orbit equivalence relations, the compact open subgroup may be `modded out' and ignored.

\begin{lemma}\label{modout} Suppose $E^X_G$ is induced by a Borel action of a Polish group $G$ on a standard Borel space $X$, and suppose $G$ has a normal subgroup $N$ such that the subequivalence relation $E^X_N$ (induced by the restriction of the action to $N\times X$) is smooth. Then $E^X_G \leq_B E^Y_{G/N}$ for some Borel $Y\subseteq X$.
\end{lemma}
\begin{proof} Since $E^X_N$ is smooth, it follows from Theorem \ref{idealistic} \& Lemma \ref{borelselector} that it has a Borel selector (which also serves as the reduction function), i.e., we have a Borel function $s:X\rightarrow X$ such that for each $x,y\in X$, $s(x)E^X_N x$ and $xE^X_N y \Rightarrow s(x)=s(y)$. Letting $Y$ be the image of the selector $s$, we define the new action $\hat{\cdot}$ of $G/N$ on $Y$ by $gN\,\hat{\cdot}\,x=s(g\cdot x)$. Note that since each $s(g\cdot x)\in Ng\cdot x = gN\cdot x$, we have that $s(g\cdot s(h\cdot x))\in gN\cdot s(h\cdot x)\subset gN\cdot (hN\cdot x) =ghN\cdot x$, and hence each $s(g\cdot s(h\cdot x))=s(gh\cdot x)$. Thus $gN\,\hat{\cdot}\,(hN\,\hat{\cdot}\,x)=ghN\,\hat{\cdot}\,x$, and $\hat{\cdot}$ is indeed an action.\end{proof}

In the case of  Polish LCA groups, we may now simplify the problem of bounding their complexity under Borel reduction.

\begin{theorem}\label{LCAuniversal} Suppose $E^X_G$ is induced by a Borel action of a Polish (i.e., second countable) LCA group $G$ on a standard Borel space $X$. Then $E^X_G$ is Borel reducible to an equivalence relation which is induced by a Borel action of $\mathbb{R}^{n}\times\mathbb{Z}^{<\omega}$, for some non-negative integer $n$, on a Borel subset $Y\subseteq X$.
\end{theorem}
\begin{proof} By Theorem \ref{structurethm}, $G\cong\mathbb{R}^{n}\times H$ where $H$ has a compact open subgroup $K$. Then since $K$ is compact and therefore acts smoothly (see $\S 5.4$ of \cite{gao}), we may apply Lemma \ref{modout} to reduce $E^X_G$ to an action of $$\frac{\mathbb{R}^{n}\times H}{\{0\}^{n}\times K}\cong\mathbb{R}^{n}\times H/K$$ on a Borel $Y\subset X$ where $H/K$ must be a countable abelian group since $H$ is second countable and $K$ is open. Finally since every countable abelian group is a homomorphic image of $\mathbb{Z}^{<\omega}$, we may let $\phi  : \mathbb{R}^{n}\times\mathbb{Z}^{<\omega}\rightarrow\mathbb{R}^{n}\times H/K$ be the appropriate quotient map, and then where $a:(\mathbb{R}^{n}\times H/K)\times Y\rightarrow Y$ is the Borel action given by Lemma \ref{modout}, we let $\mathbb{R}^{n}\times\mathbb{Z}^{<\omega}$ act on $Y$ by $\hat{a}(g,x)= a(\phi(g),x)$. Then the identity function on $Y$ is a Borel reduction from $E^Y_{\mathbb{R}^{n}\times H/K}$ to $E^{Y}_{\mathbb{R}^{n}\times\mathbb{Z}^{<\omega}}$.\end{proof}

\subsection{Reduction to $\mathbb{R}^{<\omega}$-actions}

\begin{theorem}\label{reductiontoR}Suppose $E$ is the orbit equivalence relation on a standard Borel space $X$ which is induced by a Borel action, $a:G\times X\rightarrow X$, of $$G=(\bigoplus_{\alpha}\mathbb{T})\oplus(\bigoplus_{\beta}\mathbb{R})\oplus A$$ where $\alpha,\beta$ are each finite or $\omega$ and $A$ is a countable abelian group. Then $E$ is Borel reducible to an orbit equivalence relation induced by a Borel action of $\mathbb{R}^{<\omega}$ on $X\times\mathbb{T}^{<\omega}$.
\end{theorem}
\begin{proof}
Since every countable abelian group is a homomorphic image of $\mathbb{Z}^{<\omega}$, it is clear that $(\bigoplus_{\alpha}\mathbb{T})\oplus(\bigoplus_{\beta}\mathbb{R})\oplus A$ is a quotient of $\mathbb{T}^{<\omega}\oplus\mathbb{R}^{<\omega}\oplus\mathbb{Z}^{<\omega}$. Letting $\pi:\mathbb{T}^{<\omega}\oplus\mathbb{R}^{<\omega}\oplus\mathbb{Z}^{<\omega}\rightarrow(\bigoplus_{\alpha}\mathbb{T})\oplus(\bigoplus_{\beta}\mathbb{R})\oplus A$ be the quotient map which projects $\mathbb{T}^{<\omega}$ and $\mathbb{R}^{<\omega}$ down to the appropriate number of coordinates and maps $\mathbb{Z}^{<\omega}$ onto $A$, it follows that $\pi$ is Borel and that the action $b$ of $\mathbb{T}^{<\omega}\oplus\mathbb{R}^{<\omega}\oplus\mathbb{Z}^{<\omega}$ on $X$ defined by $b((\vec{u},\vec{v},\vec{w}),x)=a(\pi(\vec{u},\vec{v},\vec{w}),x)$ is Borel and produces the same orbits as the action $a$.

Now, fix an isomorphism/reordering of coordinates $\phi:\mathbb{R}^{<\omega}\oplus\mathbb{Z}^{<\omega}\rightarrow \mathbb{R}^{<\omega}\oplus\mathbb{R}^{<\omega}\oplus\mathbb{Z}^{<\omega}$ and let $\pi_{\mathbb{Z}}:\mathbb{R}\rightarrow\mathbb{T}$ be the usual quotient map from $\mathbb{R}$ onto $\mathbb{T}=\mathbb{R}/\mathbb{Z}$. Then, we achieve the same orbits as $b$ with the action $b':(\mathbb{R}^{<\omega}\oplus\mathbb{R}^{<\omega}\oplus\mathbb{Z}^{<\omega})\times X\rightarrow X$ where $b'((\vec{u},\vec{v},\vec{w}),x)=b(((\pi_{\mathbb{Z}}(u_n))_{n\in\omega},\vec{v},\vec{w}), x)$. Then, these orbits are given by an action $(g,x)\mapsto g\cdot x$ of $\mathbb{R}^{<\omega}\oplus\mathbb{Z}^{<\omega}$ by letting $g\cdot x =b'(\phi(g),x)$ for each $g\in \mathbb{R}^{<\omega}\oplus\mathbb{Z}^{<\omega}$.

Finally, we consider the action $\hat{\cdot}$ of $\mathbb{R}^{<\omega}\oplus\mathbb{R}^{<\omega}$ on $X\times\mathbb{T}^{<\omega}$ which is defined by $$\langle\vec{u},\vec{v}\rangle\,\hat{\cdot}\,\langle x, \vec{t}\rangle=\langle\langle\vec{u},(\lfloor v_{n}+t_{n}\rfloor)_{n\in\omega}\rangle\cdot x, (v_{n}+t_{n}-\lfloor v_{n}+t_{n}\rfloor)_{n\in\omega}\rangle.$$  Note that each $t_n\in[0,1)$, so it is clear that $\langle\vec{0},\vec{0}\rangle\,\hat{\cdot}\,\langle x,\vec{t}\rangle=\langle x,\vec{t}\rangle$. Also, whenever $k$ is an integer $\lfloor r-k\rfloor=\lfloor r\rfloor -k$. So,
\begin{align*}
\langle\vec{u},\vec{v}\rangle&\,\hat{\cdot}\,(\langle\vec{w},\vec{z}\rangle\,\hat{\cdot}\,\langle x, \vec{t}\rangle)\\
&=\langle\vec{u},\vec{v}\rangle\,\hat{\cdot}\,\langle\langle\vec{w},(\lfloor z_{n}+t_{n}\rfloor)_{n\in\omega}\rangle\cdot x, (z_{n}+t_{n}-\lfloor z_{n}+t_{n}\rfloor)_{n\in\omega}\rangle\\
&=\langle\langle\vec{u}, (\lfloor v_{n}+z_{n}+t_{n}-\lfloor z_{n}+t_{n}\rfloor\rfloor)_{n\in\omega}\rangle\cdot(\langle\vec{w},(\lfloor z_{n}+t_{n}\rfloor)_{n\in\omega}\rangle\cdot x),\\
&\quad\quad\quad\quad(v_n+z_n+t_n-\lfloor z_{n}+t_{n}\rfloor-\lfloor v_n+z_n+t_n-\lfloor z_{n}+t_{n}\rfloor\rfloor)_{n\in\omega}\rangle\\
&=\langle\langle \vec{u}+\vec{w}, (\lfloor v_n+z_n+t_n\rfloor)_{n\in\omega}\rangle\cdot x,\\
&\quad\quad\quad\quad(v_n+z_n+t_n-\lfloor v_n+z_n+t_n\rfloor)_{n\in\omega}\rangle\\
&=\langle\vec{u}+\vec{w},\vec{v}+\vec{z}\rangle\,\hat{\cdot}\,\langle x, \vec{t}\rangle,
\end{align*}
and $\hat{\cdot}$ is indeed an action. Note we may use a simple isomorphism/reordering of coordinates to view the action of $\mathbb{R}^{<\omega}\oplus\mathbb{R}^{<\omega}$ on $X\times\mathbb{T}^{<\omega}$ as an action of just $\mathbb{R}^{<\omega}$ on $X\times\mathbb{T}^{<\omega}$, and we can see that $E=E^{X}_{\mathbb{R}^{<\omega}\oplus\mathbb{Z}^{<\omega}}\leq_{B} E^{X\times\mathbb{T}^{<\omega}}_{\mathbb{R}^{<\omega}}$ by the reduction $f(x)=\langle x,\vec{0}\rangle$ since $\langle\vec{u},\vec{v}\rangle\,\hat{\cdot}\,\langle x, \vec{0}\rangle=\langle\langle\vec{u},(\lfloor v_{n}\rfloor)_{n\in\omega}\rangle\cdot x, (v_{n}-\lfloor v_{n}\rfloor)_{n\in\omega}\rangle=\langle y,\vec{0}\rangle\iff \langle\vec{u},\vec{v}\rangle\in\mathbb{R}^{<\omega}\oplus\mathbb{Z}^{<\omega}$ and $\langle\vec{u},\vec{v}\rangle\cdot x=y$.\end{proof}


\section{Non-free actions of $\mathbb{R}^{<\omega}$}\label{reductiontofreesection}

In this section, we will consider the stabilizers under an $\mathbb{R}^{<\omega}$ action and produce Borel functions which provide bases for the stabilizers as well as for an algebraic complement of their $\mathbb{R}$-spans. We may then split the space into countably many disjoint invariant pieces according to the isomorphism types of the quotients of $\mathbb{R}^{<\omega}$ by these stabilizers. Finally, we will use the basis functions to produce a free action of that quotient on each of these countably many  pieces which produces the same orbits as the original action.

\subsection{Quotients and Closed Subgroups of $\mathbb{R}^{<\omega}$}\label{subgrpsofRinfty} 

We will abuse notation somewhat and say $\mathbb{R}^n$ whenever we mean the subgroup $(\prod_{i=1}^{n}\mathbb{R})\times(\prod_{i=n+1}^{\infty}\{0\})$ of $\mathbb{R}^{<\omega}$, and when we refer to a closed subgroup of any $\mathbb{R}^n$ we mean with respect to the usual topology. However, when we refer to a closed subgroup of $\mathbb{R}^{<\omega}$ we mean with respect to the topology where a set $A\subset\mathbb{R}^{<\omega}$ is closed iff each $A\cap\mathbb{R}^n$ is closed in $\mathbb{R}^n$. Equivalently, the topology on $\mathbb{R}^{<\omega}$ is the subspace topology inherited from the ``box'' topology on the product $\mathbb{R}^{\omega}$ where all products of open sets form a basis for the topology. This is also the topology on $\mathbb{R}^{<\omega}$ when viewed as the direct limit of the $\mathbb{R}^n$'s. The structure of the closed subgroups and Hausdorff quotients of $\mathbb{R}^{<\omega}$ (and in fact any of the entire class of groups considered in this paper) under this topology are explored thoroughly in \cite{brown_higgins_morris}. But we will highlight a few important points:

\begin{proposition}\label{bhm}\emph{(Brown--Higgins--Morris, \cite{brown_higgins_morris})}
\begin{itemize}
\item[(i)] Every finite-dimensional subspace of $\mathbb{R}^{<\omega}$ has the standard topology.
\item[(ii)] A subset of $\mathbb{R}^{<\omega}$ is open (closed) iff it meets each finite-dimensional subspace $F$ in an open (closed) subset of $F$.
\item[(iii)] $\mathbb{R}^{<\omega}$ is a topological vector space over $\mathbb{R}$. (i.e., the scalar multiplication is also continuous)
\item[(iv)] Any Hausdorff topological vector space of algebraic dimension $\aleph_{0}$ over $\mathbb{R}$ and having property \emph{(ii)} is isomorphic, as a topological vector space, to $\mathbb{R}^{<\omega}$.
\item[(v)] Every closed subgroup of $\mathbb{R}^{<\omega}$ is topologically isomorphic to some $(\bigoplus_{\alpha}\mathbb{R})\oplus(\bigoplus_{\beta}\mathbb{Z})$ where $\alpha,\beta$ are each finite or $\omega$, and every Hausdorff quotient of $\mathbb{R}^{<\omega}$ is topologically isomorphic to some $(\bigoplus_{\beta}\mathbb{T})\oplus(\bigoplus_{\gamma}\mathbb{R})$ where $\beta,\gamma$ are each finite or $\omega$.
\end{itemize}\end{proposition}

Property (v)  of Proposition \ref{bhm} is the most important for our purposes, and in particular, it was proven by showing the following:
\begin{proposition}\label{basis}\emph{(Brown--Higgins--Morris)} If $G$ is a closed subgroup of $\mathbb{R}^{<\omega}$, then there is an $\mathbb{R}$-basis for $\mathbb{R}^{<\omega}$ of the form $$\{\vec{u_i}:i\in\alpha\}\cup\{\vec{v_i}:i\in\beta\}\cup\{\vec{w_i}:i\in\gamma\}$$ where $\alpha,\beta,\gamma$ are each finite or $\omega$, such that
\begin{itemize}
\item[(i)]$\spn_{\mathbb{R}}\{\vec{u_i}:i\in\alpha\}$ is the largest vector subspace of $G$,
\item[(ii)]$G\cap\spn_{\mathbb{R}}\{\vec{v_i}:i\in\beta\}=\spn_{\mathbb{Z}}\{\vec{v_i}:i\in\beta\}$, and 
\item[(iii)]$G\cap\spn_{\mathbb{R}}\{\vec{w_i}:i\in\gamma\}=\{\vec{0}\}$.
\end{itemize}\end{proposition}

The reason this is useful here is that the potential stabilizers of our Borel $\mathbb{R}^{<\omega}$ action are precisely these closed subgroups of $\mathbb{R}^{<\omega}$, and points belonging to the same orbit must have the same stabilizer. To see this note that, because $\mathbb{R}^{<\omega}$ is abelian, $x E y$ implies the stabilizers $G_{x}=G_{y}$. And each $\mathbb{R}^n$ is a Polish group acting in a Borel manner on $X$ (by the restriction of $\mathbb{R}^{<\omega}$'s action), so by Corollary \ref{closedstabilizers} each $G_{x}\cap\mathbb{R}^n$ is a closed subgroup of $\mathbb{R}^n$. We will show that, given a point $x$, there is an invariant Borel way to construct the basis in Proposition \ref{basis} for $G_x$ and then this will allow us to define a free action of $(\bigoplus_{\beta}\mathbb{T})\oplus(\bigoplus_{\gamma}\mathbb{R})$ on $[x]_E$.

\subsection{Constructing the Bases for the Stabilizers} 
Let $G_x^n$ denote the stabilizer of $x$ under the restriction of the action to the $\mathbb{R}^n$ subspace, and note that each $G_x^{n}=G_{x}\cap\mathbb{R}^n$. Also, let $U_{x}^{n}=\{\vec{g}\in G^n_{x}:\forall q\in\mathbb{Q}(q\!\vec{g}\in G^n_{x})\}$, and note that this definition implies each $U_x^{n}$ is closed, each $U_x^{n}$ is the largest vector subspace of $\mathbb{R}^n$ contained in the subgroup $G_x^{n}$, and $U_{x}=\bigcup_{n=1}^{\infty}U_x^{n}$ is the closed subgroup of $\mathbb{R}^{<\omega}$ which is also the largest vector subspace of $\mathbb{R}^{<\omega}$ contained in $G_x$. We will now show that we can construct the bases of Proposition \ref{basis} for $G_x$ effectively from $x$.

First we need to check that we have the following:
\begin{lemma}\label{boreltosubspace}\emph{(S. Jackson, personal communication)}
For each $n$ the map $x\mapsto U^n_x$ from $X$ to $F(\mathbb{R}^{n})$ is Borel.
\end{lemma}
\begin{proof} By the definition of the Effros structure on $F(\mathbb{R}^{n})$, it suffices to show that for each open $O\subset \mathbb{R}^{n}$ the set $\{x:U^{n}_{x}\cap O\neq\emptyset\}$ is Borel. Let $\{B_{k}\}_{k=1}^{\infty}$ be an enumeration of the basis for $\mathbb{R}^{n}$ consisting of the products of bounded open intervals with rational endpoints, let $A\subset\mathbb{N}$ be the set of all $k\in\mathbb{N}$ such that $\bar{B}_k\subset O$, fix the sup norm $||\vec{g}||=||(g_1,g_2,\ldots,g_n)||=\max\{|g_{i}| : 1\leq i\leq n\}$ on the vectors of $\mathbb{R}^n$, and let $\rho:\mathbb{R}^n\times\mathbb{R}^n \rightarrow [0,\infty)$ be the induced metric given by $\rho(\vec{g},\vec{h})=||\vec{g}-\vec{h}||$. We also let, as usual, $\rho(\vec{g},F)=\inf\{\rho(\vec{g},\vec{h}):\vec{h}\in F\}$ for a closed $F$. Now, we claim that 
\begin{equation*}U^{n}_{x}\cap O\neq\emptyset\Leftrightarrow\exists k\in A\,\forall m\forall F\in[\mathbb{Q}]^{<\omega}\exists \vec{z}\in\mathbb{Q}^{n}[\vec{z}\in B_{k} \,\&\, \forall q\in F(\rho(q\!\vec{z},G^{n}_{x})<\frac{1}{m})].\end{equation*}
For the forward direction, suppose $\vec{g}\in U^{n}_{x}\cap O$ and then let $B_k$ be a basic open set such that $\vec{g}\in B_{k}\subset \bar{B}_k\subset O$. Then, for any finite subset $F\subset\mathbb{Q}$ and $m\in\mathbb{N}$, since $\mathbb{Q}^{n}$ is dense in $\mathbb{R}^{n}$ we may let $\vec{z}\in\mathbb{Q}^n\cap B_k$ be s.t. $\rho(\vec{z},\vec{g})<\frac{1}{mp}$ where $p=\max\{1,\max\{|q|:q\in F\}\}$. Then, since $\vec{g}\in U^n_{x}$ and hence every $q\!\vec{g}\in G^n_{x}$, it follows that for each $q\in F$ we have $\rho(q\!\vec{z},G^n_{x})\leq\rho(q\!\vec{z},q\!\vec{g})=\max\{|qz_{i}-qg_{i}| : 1\leq i\leq n\}\leq p(\max\{|z_{i}-g_{i}| : 1\leq i\leq n\})=p(\rho(\vec{z},\vec{g}))<\frac{p}{mp}=\frac{1}{m}$.

For the converse, let $\{q_i\}_{i=1}^{\infty}$ be an enumeration of the rationals with $q_{1}=1$, and suppose we have $B_{k}\subset \bar{B}_k\subset O$ such that \begin{equation*}\forall m\forall F\in[\mathbb{Q}]^{<\omega}\exists \vec{z}\in\mathbb{Q}^{n}[\vec{z}\in B_{k} \,\&\, \forall q\in F(\rho(q\!\vec{z},G^{n}_{x})<\frac{1}{m})].\end{equation*} Then we have a sequence $\{\vec{z}\!\!_{m}\}_{m=1}^{\infty}$ such that each $\vec{z}\!\!_{m}\in B_{k}$ and $i\leq m\Rightarrow\rho(q_{i}\!\vec{z}\!\!_{m},G^{n}_{x})<\frac{1}{m}$. In particular, $\{\vec{z}\!\!_{m}\}_{m=1}^{\infty}\subset\bar{B}_k$ which is compact and so there must be a subsequence $\{\vec{z}\!\!_{m_j}\}_{j=1}^{\infty} $converging to some $\vec{g}$ which belongs to $O$ since $\bar{B}_k\subset O$ and also to $G_x^n$ since each $\rho(\vec{z}\!\!_{m_j},G^{n}_{x})<\frac{1}{m_j}$ and hence $\rho(\vec{g},G^{n}_{x})=0$. But also for each $q_i$ we have  $\rho(q_{i}\!\vec{z}\!\!_{m_j},G^{n}_{x})<\frac{1}{m_j}$ for all large enough $j$. Hence each $\rho(q_{i}\!\vec{g},G^{n}_{x})=0$ so $\forall q\in\mathbb{Q}(q\!\vec{g}\in G^n_x)$ and $\vec{g}\in U^n_x$.

Since by Theorem \ref{boreltostabilizer} we already know $x\mapsto G^n_x$ is Borel, the claim shows that $x\mapsto U^n_x$ is also Borel.\end{proof}

Now we begin by constructing from $x$ a basis for the vector subspace part of $G_x$.

\begin{lemma}\label{ubasis}Given a Borel action of $\mathbb{R}^{<\omega}$ on a standard Borel space $X$, there are invariant Borel maps $x\mapsto\alpha(x)\in\omega+1$ and $x\mapsto \{\vec{u_i}(x):i\in\alpha(x)\}\in\bigcup_{j\leq\omega}(\mathbb{R}^{<\omega})^j$ such that $\{\vec{u_i}(x):i\in\alpha(x)\}$ is a basis for $U_x$.
\end{lemma}
\begin{proof} We prove the result by defining a procedure for the construction of such a basis. For each $n\in\mathbb{N}$, let $\{f^n_{k}:F(\mathbb{R}^n)\setminus\emptyset\rightarrow\mathbb{R}^{n}\,|\,k\in\mathbb{N}\}$ be the sequence of selection functions given by Theorem \ref{selection}. Also note that, given a finite sequence $\{\vec{y_i}\in\mathbb{R}^{n}:i\leq k\}$, then \begin{align*}\vec{h}\in\spn_{\mathbb{R}}(\{\vec{y_i}\}_{i\leq k})&\Leftrightarrow d(\vec{h},\spn_{\mathbb{Q}}(\{\vec{y_i}\}_{i\leq k}))=0\\
&\Leftrightarrow \forall p\in\mathbb{N}\,\exists \vec{q}\in\mathbb{Q}^{k}[ d(\vec{h}, \sum_{i\leq k}\! q_{i}\vec{y_i}) <\frac{1}{p} ],\end{align*}so $\vec{h}\in\spn_{\mathbb{R}}(\{\vec{y_i}\}_{i\leq k})$ and $\vec{h}\nin\spn_{\mathbb{R}}(\{\vec{y_i}\}_{i\leq k})$ are Borel statements. 

Now, we use the selection functions to find new $\mathbb{R}$-linearly independent vectors to add to the basis.

\noindent For $n=1$,
\begin{itemize}
\item If $\exists k\in\mathbb{N}[f^{1}_{k}(U^{1}_{x})\neq \vec{0}]$, let $k_{0}\in\mathbb{N}$ be least such that $f^{1}_{k_0}(U^{1}_{x})\neq \vec{0}$, let $\vec{u_0}(x)=f^{1}_{k_0}(U^{1}_x)$, and let $\alpha_{1}(x)=1$.
\item Otherwise, let $\alpha_{1}(x)=0$. Then the $n=1$ step halts and we move on to $n=2$.
\end{itemize}
For $n=m+1$, let $\{\vec{u_i}(x):i\in\alpha_{m}(x)\}$ be the sequence constructed from the previous $m$ dimensions. Then,
\begin{itemize}
\item If $\forall k\in\mathbb{N}\,[f^{m+1}_{k}(U^{m+1}_{x})\in\spn_{\mathbb{R}}\{\vec{u_i}(x):i\in\alpha_{m}(x)\}]$, then we let $\alpha_{m+1}(x)=\alpha_{m}(x)$, the $n=m+1$ step halts, and we move on to $n=m+2$.

\item If $\exists k\in\mathbb{N}\,[f^{m+1}_{k}(U^{m+1}_{x})\nin\spn_{\mathbb{R}}\{\vec{u_i}(x):i\in\alpha_{m}(x)\}]$, then we extend the sequence as follows:
\begin{itemize}
\item Let $k_1$ be least such that $f^{m+1}_{k_{1}}(U^{m+1}_{x})\nin\spn_{\mathbb{R}}\{\vec{u_i}(x):i\in\alpha_{m}(x)\}$, and let $\vec{u}\!_{\alpha_{m}(x)+1}(x)=f^{m+1}_{k_{1}}(U^{m+1}_{x})$.
\item If $\exists k\in\mathbb{N}\,[f^{m+1}_{k}(U^{m+1}_{x})\nin\spn_{\mathbb{R}}\{\vec{u_i}(x):i\in\alpha_{m}(x)+j\}]$, we let $k_{j+1}$ be least such that $f^{m+1}_{k_{j+1}}(U^{m+1}_{x})\nin\spn_{\mathbb{R}}\{\vec{u_i}(x):i\in\alpha_{m}(x)+j\}$, and let $\vec{u}\!_{\alpha_{m}(x)+j+1}(x)=f^{m+1}_{k_{j+1}}(U^{m+1}_{x})$.
\item Once $\vec{u}\!_{\alpha_{m}(x)+l}(x)$ is constructed so that $\forall k\in\mathbb{N}\,[f^{m+1}_{k}(U^{m+1}_{x})\in\spn_{\mathbb{R}}\{\vec{u_i}(x):i\in\alpha_{m}(x)+l\}]$, let $\alpha_{m+1}(x)=\alpha_{m}(x)+l$. Then the $n=m+1$ step halts and we move on to $n=m+2$.
\end{itemize}\end{itemize}
Finally, let $\alpha(x)=\sup\{\alpha_{n}(x):n\in\mathbb{N}\}$.

Since, by Theorem \ref{boreltostabilizer} each $x\mapsto G^n_x$ is Borel, and by Lemma \ref{boreltosubspace} each $x\mapsto U^n_x$ is Borel, this construction provides a Borel map $x\mapsto \{\vec{u_i}(x):i\in\alpha(x)\}$ where $\{\vec{u_i}(x):i\in\alpha(x)\}$ is selected to be a basis for $U_x$ as desired. To see this, suppose for contradiction that $\vec{h}\in U_x$ but $\vec{h}\nin\spn_{\mathbb{R}}\{\vec{u_i}(x):i\in\alpha(x)\}$. Then for some $m\in\mathbb{N}$ we would have $\vec{h}\in U_x^m$ but $\vec{h}\nin\spn_{\mathbb{R}}\{\vec{u_i}(x):i\in\alpha_{m}(x)\}$. Since $\spn_{\mathbb{R}}\{\vec{u_i}(x):i\in\alpha_{m}(x)\}$ is closed in $\mathbb{R}^m$, it follows that there would exits an open neighborhood $\vec{h}\in O\subset\mathbb{R}^m$ which is disjoint from $\spn_{\mathbb{R}}\{\vec{u_i}(x):i\in\alpha_{m}(x)\}$. But if that were the case, $O$ must contain some $f^{m}_{k}(U^{m}_{x})\nin\spn_{\mathbb{R}}\{\vec{u_i}(x):i\in\alpha_{m}(x)\}$ and the $m$th stage of the construction would not have halted.\end{proof}

Next, we will construct a $\mathbb{Z}$-basis for the discrete part of $G_x$ so that when combined with the basis for the vector subspace part of $G_x$ forms an $\mathbb{R}$-basis for $\spn_{\mathbb{R}}(G_x)$.

\begin{lemma}\label{vbasis}Given a Borel action of $\mathbb{R}^{<\omega}$ on a standard Borel space $X$, there are invariant Borel maps $x\mapsto\beta(x)\in\omega+1$ and $x\mapsto \{\vec{v_i}(x):i\in\beta(x)\}\in\bigcup_{j\leq\omega}(\mathbb{R}^{<\omega})^j$ such that $\{\vec{v_i}(x):i\in\beta(x)\}$ is linearly independent over $\mathbb{R}$ and $$G_{x}=U_{x}\oplus\bigoplus_{i\in\beta(x)}\mathbb{Z}\vec{v_i}(x).$$
\end{lemma}
\begin{proof}  Let the basis $\{\vec{u_i}(x):i\in\alpha(x)\}$ for $U_x$ be constructed as in Lemma \ref{ubasis}. First, we construct a basis for a complementary subspace $V_{x}$ of $U_{x}$ in $\mathbb{R}^{<\omega}$.

Given any finite sequence $\{\vec{y_i}\in\mathbb{R}^{<\omega}:i\leq k\}$,
\begin{align*}\vec{h}\in\spn_{\mathbb{R}}&(U_{x}\cup\{\vec{y_j}\}_{j\leq k})\\
&\Leftrightarrow \vec{h}\in\spn_{\mathbb{R}}\{\vec{u_i}(x):i\in\alpha(x)\}+\spn_{\mathbb{R}}\{\vec{y_j}:j\leq k\}\\
&\Leftrightarrow \exists n\,[\vec{h}\in\spn_{\mathbb{R}}\{\vec{u_i}(x):i\in\alpha_{n}(x)\}+\spn_{\mathbb{R}}\{\vec{y_j}:j\leq k\}]\\
&\Leftrightarrow \exists n\,[d(\vec{h}, \spn_{\mathbb{Q}}\{\vec{u_i}(x):i\in\alpha_{n}(x)\}\!+\!\spn_{\mathbb{Q}}\{\vec{y_j}:j\leq k\})\!=\!0]\\
&\Leftrightarrow \exists n\,\forall p\in\mathbb{N}\,\exists\vec{q}\in\mathbb{Q}^{\alpha_{n}(x)+k}\,[d(\vec{h},\!\!\sum_{i\in\alpha_{n}(x)}\!\!q_{i}\!\vec{u_i}+\sum_{j\leq k}\!\!q_{\alpha_{n}(x)+j}\!\vec{y_j}\,)<\frac{1}{p}],\end{align*}
so $\vec{h}\in\spn(U_{x}\cup\{\vec{y_j}\}_{j\leq k})$ and $\vec{h}\nin\spn(U_{x},\{\vec{y_j}\}_{j\leq k})$ are Borel statements. Now, let  each $\vec{e}\!\!_{k}\in\mathbb{R}^{<\omega}$ for $k\in\mathbb{N}$ be the usual coordinate vector where $(\vec{e}\!\!_{k})_{k}=1$ and $(\vec{e}\!\!_{k})_{j}=0$ for all $j\neq k$. We construct a subset $A_{x}\subset\mathbb{N}$ inductively as follows:
\begin{itemize}
\item If $\forall k\in\mathbb{N}\,[\vec{e}\!\!_{k}\in \spn_{\mathbb{R}}(U_x)]$ let $A_{x}=\emptyset$.
\item If $\exists k\in\mathbb{N}\,[\vec{e}\!\!_{k}\nin \spn_{\mathbb{R}}(U_x)]$, then 
\begin{itemize} \item let $k_0$ be the least such $k$, and
\item if $\exists k> k_{i}\,[\vec{e}\!\!_{k}\nin \spn_{\mathbb{R}}(U_{x}\cup\{\vec{e}\!_{k_j}:j\leq i\})]$, let $k_{i+1}$ be the least such $k$, 
\item if $\forall k> k_{i}\,[\vec{e}\!\!_{k}\in \spn_{\mathbb{R}}(U_{x}\cup\{\vec{e}\!_{k_j}:j\leq i\})]$, then the construction halts and we let $A_{x}$ be the set of $k_{j}$ for $j\leq i$.
\item if the construction does not halt then we let $A_{x}=\{k_{i}:i\in\mathbb{N}\}$.\end{itemize}\end{itemize}

We let $V_{x}=\spn_{\mathbb{R}}\{\vec{e}\!_{k}:k\in A_{x}\}$. That $U_{x}\cap V_{x}=\{\vec{0}\}$ is clear. And $\vec{h}\in\mathbb{R}^{<\omega}\Rightarrow\exists n(\vec{h}\in\mathbb{R}^{n})\Rightarrow\vec{h}\in\spn_{\mathbb{R}}(\{\vec{e}\!\!_{k}:k\leq n\})$ where each of these $\vec{e}\!\!_{k}$'s is an element of $\spn_{\mathbb{R}}(U_{x}\cup\{\vec{e}\!_{k}:k\in A_{x}\})$. So $\mathbb{R}^{<\omega}=U_{x}\oplus V_{x}$.

Now we let $D_{x}=G_{x}\cap V_{x}$ and let each $D^n_{x}=D_{x}\cap\mathbb{R}^n$. Then each $D^n_{x}=G^n_{x}\cap\spn_{\mathbb{R}}\{\vec{e}\!_{k}:k\in A_{x}\,\,\&\,\,k\leq n\}$. Hence each $D^n_x$ is a closed subgroup of $\mathbb{R}^n$, and $D_{x}$ is a closed subgroup of $\mathbb{R}^{<\omega}$ in the topology discussed at the beginning of Section \ref{subgrpsofRinfty}. Also, each $D^n_x$ must be discrete in $\mathbb{R}^{n}$. To see this, note that $D^n_{x}\subset G^n_{x}$ cannot contain any whole line passing through $\vec{0}$ since then that line would be contained in $U_{x}$. Since every non-discrete closed subgroup of $\mathbb{R}^{n}$ must contain a whole line through the origin (See Ch.VII, $\S 1$, Proposition 3 in \cite{bourbaki}), this means $D^n_{x}$ must be discrete. But if $D^n_{x}$ is closed discrete then every one of its subsets is closed. Hence every subset $F\subset D_{x}$ is such that each $F\cap\mathbb{R}^{n}$ is closed in $\mathbb{R}^n$. So all subsets of $D_{x}$ are closed in $\mathbb{R}^{<\omega}$, and it follows that $D_{x}$ is a discrete subgroup of $\mathbb{R}^{<\omega}$. By Proposition 4 of \cite{brown_higgins_morris} and its proof, it follows that $D_{x}$ is topologically isomorphic to either $\mathbb{Z}^{<\omega}$ or $\mathbb{Z}^{n}$ for some $n$ and has a $\mathbb{Z}$-basis which is linearly independent over $\mathbb{R}$. Moreover, they show that the basis may be constructed inductively, i.e., given a $\mathbb{Z}$-basis for $D^n_x$ there exists a $\mathbb{Z}$-basis for $D^{n+1}_x$ which contains the $D^n_x$ basis. Also they get for free that the $\mathbb{Z}$-basis is linearly independent over $\mathbb{R}$ because any subset of an $\mathbb{R}^n$ with a discrete span which is linearly independent over $\mathbb{Z}$ must also be linearly independent over $\mathbb{R}$. (It follows from Ch.VII, $\S 1$, Proposition 1 of \cite{bourbaki} that if $\spn_{\mathbb{Z}}\{\vec{a_i}:i\leq p\}$ is discrete and $\sum_{i\leq p}r_{i}\vec{a_{i}}=\vec{0}$ then there must be some rational combination $\sum_{i\leq p}q_{i}\vec{a_{i}}=\vec{0}$. But then by multiplying by a common denominator this would contradict the $\mathbb{Z}$-independence.) Thus, it suffices now to provide a Borel procedure for selecting the extensions of the bases for the $D^n_x$'s, which we know to exist, as we induct up the dimension.

To do this we will need to know that, for each $n$, the map $x\mapsto D^n_x$ from $X$ to $F(\mathbb{R}^{n})$ is Borel. So for each open $O\subset \mathbb{R}^{n}$ we want to check that $D^n_{x}\cap O\neq\emptyset$ is a Borel statement. Let $\{B_{k}\}_{k=1}^{\infty}$ be an enumeration of the basis for $\mathbb{R}^{n}$ consisting of the products of bounded open intervals with rational endpoints, and let $N_{O}\subset\mathbb{N}$ be the set of all $k\in\mathbb{N}$ such that $\bar{B}_k\subset O$. Also, let $A^n_{x}=\{k\in\mathbb{N}:k\in A_{x}\,\,\&\,\,k\leq n\}$. We claim that 
\begin{equation*}D^{n}_{x}\cap O\!\neq\!\emptyset\Leftrightarrow\exists k\in N_{O}\forall m\in\mathbb{N}\,\exists\vec{q}\in\mathbb{Q}^{A^n_x}[\sum_{i\in A^n_x}q_{i}\!\vec{e}\!_{i}\in B_{k} \,\&\, d(\sum_{i\in A^n_x}q_{i}\!\vec{e}\!\!_{i}\,, \,G^n_x)<\frac{1}{m}].\end{equation*}
For the forward direction, if $\vec{g}\in D^{n}_{x}\cap O=G^n_{x}\cap (V_{x}\cap\mathbb{R}^{n})\cap O$ we may choose any basic open $B_k$ with $\vec{g}\in B_{k}\subset \bar{B}_k\subset O$. Then since $\spn_{\mathbb{Q}}\{\vec{e}\!_{i}: i\in A^n_{x}\}$ is dense in $V_{x}\cap\mathbb{R}^{n}$, we may always find a sum $\sum_{i\in A^n_x}q_{i}\!\vec{e}$ in $\spn_{\mathbb{Q}}\{\vec{e}\!_{i}: i\in A^n_{x}\}$ which is within some $\epsilon$ neighborhood of $\vec{g}$ that is contained in $B_k$ and where $\epsilon<1/m$. For the converse, note that the statement says we have a $B_{k}\subset \bar{B}_k\subset O$ such that $d(B_{k}\cap\spn_{\mathbb{Q}}\{\vec{e}\!_{i}: i\in A^n_{x}\},\,G^n_{x})=0$. Hence we may choose a sequence $\{\vec{z}\!\!_{m}\}_{m=1}^{\infty}\subset B_{k}\cap\spn_{\mathbb{Q}}\{\vec{e}\!_{i}: i\in A^n_{x}\}$ where each $d(\vec{z}\!\!_{m},\,G^n_{x})<\frac{1}{m}$. Then $\bar{B}_k$ is compact and so there must be a subsequence $\{\vec{z}\!\!_{m_j}\}_{j=1}^{\infty} $ converging to some $\vec{g}$ which belongs to $O$ since $\bar{B}_k\subset O$, belongs to $G_x^n$ since each $d(\vec{z}\!\!_{m_j},G^{n}_{x})<\frac{1}{m_j}$ and hence $d(\vec{g},G^{n}_{x})=0$, and belongs to $V_{x}\cap\mathbb{R}^n$ since $\vec{z}\!\!_{m_j}\rightarrow \vec{g}$ implies each $g_{i}$ must also be zero whenever $i\nin A^n_{x}$.

Finally, we may define a procedure for constructing the desired basis. Again let each $\{f^n_{k}:F(\mathbb{R}^n)\setminus\emptyset\rightarrow\mathbb{R}^{n}\,|\,k\in\mathbb{N}\}$ be the sequence of selection functions given by Theorem \ref{selection}. Note that since each $D^n_{x}$ is discrete and $\{f^n_{k}(D^n_{x}):k\in\mathbb{N}\}$ is dense in $D^n_x$, we have that $D^n_{x}=\{f^n_{k}(D^n_{x}):k\in\mathbb{N}\}$. So, to say that a finite sequence $\{\vec{y_i}\in D^n_{x}:i\leq m\}$ is a $\mathbb{Z}$-basis for $D^n_{x}$ is equivalent to $$\forall k\,\exists\vec{z}\in\mathbb{Z}^{m}\,[\,f^n_{k}(D^n_{x})=\sum_{i\in m}z_{i}\!\vec{y_{i}}\,]\,\,\&\,\,\forall\vec{z}\in\mathbb{Z}^{m}\,[\,\sum_{i\in m}z_{i}\!\vec{y_{i}}= 0\Rightarrow \vec{z}=\vec{0}\,]$$ and hence is a Borel statement. Now, for each $n,j\in\mathbb{N}$, let $\phi^n_{j}:\mathbb{N}\rightarrow\{f^n_{k}(D^n_{x}):k\in\mathbb{N}\}^{j}$ be an enumeration of the $j$-tuples of the $f^n_{k}(D^n_{x})$'s. We select the basis as follows:

For $n=1$,
\begin{itemize}
\item If $\exists k\in\mathbb{N}[f^{1}_{k}(D^{1}_{x})\neq \vec{0}]$, let $k_0$ be least such that $\{f^{1}_{k_0}(D^{1}_{x})\}$ is a $\mathbb{Z}$-basis for $D^1_{x}$ and let $\vec{v_0}(x)=f^{1}_{k_0}(D^{1}_{x})$ and $\beta_{1}(x)=1$.
\item If $\forall k\in\mathbb{N}[f^{1}_{k}(D^{1}_{x})= \vec{0}]$, let $\beta_{1}(x)=0$. Then the $n=1$ step halts and we move on to $n=2$.
\end{itemize}
(Note that, since our new basis vectors are independent from the $\vec{u_i}(x)$'s and also $\mathbb{R}$-linearly independent, the most new vectors that may be added at any $n$th stage is $n-\alpha_{n}(x)-\beta_{n-1}(x)$.)

For $n=m+1$, let $\{\vec{v_i}(x):i\in\beta_{m}(x)\}$ be the sequence constructed from the previous $m$ dimensions. Then,
\begin{itemize}
\item If $\{\vec{v_i}(x):i\in\beta_{m}(x)\}$ is a $\mathbb{Z}$-basis for $D^{m+1}_x$, let $\beta_{m+1}(x)=\beta_{m}(x)$. Then the $n=m+1$ step halts and we move on to $n=m+2$.

\item If $\{\vec{v_i}(x):i\in\beta_{m}(x)\}$ is not a $\mathbb{Z}$-basis for $D^{m+1}_x$, let $j$ be the unique natural $\leq n-\alpha_{m+1}(x)-\beta_{m}(x)$ such that $\exists l\in\mathbb{N}\,[\{\vec{v_i}(x):i\in\beta_{m}(x)\}\cup\phi^n_{j}(l)$ is a $\mathbb{Z}$-basis for $D^{m+1}_x\,]$, let $l_{0}$ be the least such $l$, let $\beta_{m+1}(x)=\beta_{m}(x)+j$, and let $\{\vec{v_i}(x):i\in\beta_{m+1}(x)\}=\{\vec{v_i}(x):i\in\beta_{m}(x)\}^{\frown}\{(\phi^n_{j}(l_0))_{i}:i\leq j\}$.
\end{itemize}

Finally, let $\beta(x)=\sup\{\beta_{n}(x):n\in\mathbb{N}\}$.

Since we have shown that  each $x\mapsto D^n_{x}$ is Borel, and that the necessary extensions of the $\mathbb{Z}$-basis from $D^n_{x}$ to $D^{n+1}_x$ exist, this construction provides a Borel map $x\mapsto \{\vec{v_i}(x):i\in\beta(x)\}$ where $\{\vec{v_i}(x):i\in\beta(x)\}$ is a $\mathbb{Z}$-basis (therefore also linearly independent over $\mathbb{R}$) for $G_{x}\cap V_{x}$ where $\mathbb{R}^{<\omega}=U_x\oplus V_x$. Hence $G_x$ is equal to the internal sum $(\bigoplus_{i\in\alpha(x)}\mathbb{R}\!\vec{u_i})\oplus(\bigoplus_{i\in\beta(x)}\mathbb{Z}\!\vec{v_i})$.
\end{proof}

Now we may construct a basis for a closed subspace complementary to $\spn_{\mathbb{R}}(G_x)$.
\begin{lemma}\label{wbasis}Given a Borel action of $\mathbb{R}^{<\omega}$ on a standard Borel space $X$, there are invariant Borel maps $x\mapsto\gamma(x)\in\omega+1$ and $x\mapsto \{\vec{w_i}(x):i\in\gamma(x)\}\in\bigcup_{j\leq\omega}(\mathbb{R}^{<\omega})^j$ such that $\{\vec{w_i}(x):i\in\gamma(x)\}$ is a linearly independent set and $$\mathbb{R}^{<\omega}=\spn_{\mathbb{R}}(G_{x})\oplus(\bigoplus_{i\in\gamma(x)}\!\!\mathbb{R}\!\vec{w_i}).$$
\end{lemma}
\begin{proof} Let $\{\vec{u_i}(x):i\in\alpha(x)\}$ and $\{\vec{v_i}(x):i\in\beta(x)\}$ be constructed as in the proofs of Lemmas \ref{ubasis} and \ref{vbasis}, and again let each $\vec{e}\!\!_{k}\in\mathbb{R}^{<\omega}$ for $k\in\mathbb{N}$ be the usual coordinate vector where $(\vec{e}\!\!_{k})_{k}=1$ and $(\vec{e}\!\!_{k})_{j}=0$ for all $j\neq k$. Note that, given any finite sequence $\{\vec{y_i}\in\mathbb{R}^{n}:i\leq k\}$, then
\begin{align*}\vec{h}\in&\spn_{\mathbb{R}}(G_{x}\cup\{\vec{y_i}\in\mathbb{R}^{n}:i\leq k\})\\
&\Leftrightarrow \exists n\,[\vec{h}\in\spn_{\mathbb{R}}(\{\vec{u_i}(x)\}_{i\in\alpha_{n}(x)}\cup\{\vec{v_i}(x)\}_{i\in\beta_{n}(x)}\cup\{\vec{y_i}\}_{i\leq k})]\\
&\Leftrightarrow\exists n\, [d(\vec{h},\spn_{\mathbb{Q}}(\{\vec{u_i}(x)\}_{i\in\alpha_{n}(x)}\cup\{\vec{v_i}(x)\}_{i\in\beta_{n}(x)}\cup\{\vec{y_i}\}_{i\leq k}))=0]\\
&\Leftrightarrow \exists n\,\forall l\in\mathbb{N}\,\exists \vec{q}\in\mathbb{Q}^{\alpha_{n}}\exists\vec{p}\in\mathbb{Q}^{\beta_{n}} \exists\vec{r}\in\mathbb{Q}^{k}\\&\hspace{.5in} [\, d(\vec{h}, \!\sum_{i\in\alpha_{n}(x)}\!\! q_{i}\vec{u_i}+\!\!\sum_{i\in\beta_{n}(x)}\!\! p_{i}\vec{v_i}+\sum_{i\leq k} r_{i}\vec{y_i}) <\frac{1}{l}\, ],\end{align*}
and we have that $\vec{h}\in\spn_{\mathbb{R}}(G_x)$ and $\vec{h}\nin\spn_{\mathbb{R}}(G_x)$ are Borel statements.

We construct the $\vec{w_i}(x)$'s as follows:\begin{itemize}
\item If $\forall k\in\mathbb{N}\,[\vec{e}\!\!_{k}\in \spn_{\mathbb{R}}(G_x)]$ let $\gamma(x)=0$ and $\{\vec{w_i}(x):i\in\gamma(x)\}=\emptyset$.
\item If $\exists k\in\mathbb{N}\,[\vec{e}\!\!_{k}\nin \spn_{\mathbb{R}}(G_x)]$, then 
\begin{itemize} \item let $k_0$ be the least such $k$, and let $\vec{w_0}(x)=\vec{e}\!\!_{k_0}$,
\item if $\exists k> k_{i}\,[\vec{e}\!\!_{k}\nin \spn_{\mathbb{R}}(G_{x}\cup\{\vec{w_j}(x):j\leq i\})]$, let $k_{i+1}$ be the least such $k$, and let $\vec{w}\!_{i+1}(x)=\vec{e}\!\!_{k_{i+1}}$, 
\item if $\forall k> k_{i}\,[\vec{e}\!\!_{k}\in \spn_{\mathbb{R}}(G_{x}\cup\{\vec{w_j}(x):j\leq i\})]$, then $\gamma(x)=i+1$ and the construction halts.
\item if the construction does not halt, then $\gamma(x)=\omega$.\end{itemize}\end{itemize}

That $\spn_{\mathbb{R}}(G_{x})\cap\spn_{\mathbb{R}}\{\vec{w_i}:i\in\gamma(x)\}=\{\vec{0}\}$ and the $\vec{w_i}$'s are independent is clear. And $\vec{h}\in\mathbb{R}^{<\omega}\Rightarrow\exists n(\vec{h}\in\mathbb{R}^{n})\Rightarrow\vec{h}\in\spn_{\mathbb{R}}(\{\vec{e}\!\!_{k}:k\leq n\})$ where each of these $\vec{e}\!\!_{k}$'s is an element of $\spn_{\mathbb{R}}(G_{x}\cup\{\vec{w_i}:i\in\gamma(x)\})$. Hence $\mathbb{R}^{<\omega}\subseteq$ (and therefore $=$) $\spn_{\mathbb{R}}(G_{x})\oplus\spn_{\mathbb{R}}\{\vec{w_i}:i\in\gamma(x)\}$.
\end{proof}

\subsection{Reduction to Free Actions of Countable Sums of $\mathbb{R}$ and $\mathbb{T}$}
Combining Lemmas \ref{ubasis}, \ref{vbasis}, and \ref{wbasis} we have proven the following refinement of Proposition \ref{basis}:

\begin{proposition}\label{effbasis} Given a Borel action of $\mathbb{R}^{<\omega}$ on a standard Borel space $X$, there are invariant Borel maps which provide, for each $x\in X$, $$\alpha(x),\beta(x),\gamma(x)\in\omega+1$$ and $$\{\vec{u_i}(x)\}_{i\in\alpha(x)}, \{\vec{v_i}(x)\}_{i\in\beta(x)}, \{\vec{w_i}(x)\}_{i\in\gamma(x)}\in\bigcup_{j\leq\omega}(\mathbb{R}^{<\omega})^j$$ such that $$\mathbb{R}^{<\omega}=(\bigoplus_{i\in\alpha(x)}\!\!\mathbb{R}\!\vec{u_i})\oplus(\bigoplus_{i\in\beta(x)}\!\!\mathbb{R}\!\vec{v_i})\oplus(\bigoplus_{i\in\gamma(x)}\!\!\mathbb{R}\!\vec{w_i})$$ and $$G_{x}=(\bigoplus_{i\in\alpha(x)}\!\!\mathbb{R}\!\vec{u_i})\oplus(\bigoplus_{i\in\beta(x)}\!\!\mathbb{Z}\!\vec{v_i}).\\$$
\end{proposition}

\begin{remark} It follows immediately from Proposition \ref{effbasis} that $$\mathbb{R}^{<\omega}/G_{x}\cong(\bigoplus_{i\in\beta(x)}\!\!\mathbb{T}\!\vec{v_i})\oplus(\bigoplus_{i\in\gamma(x)}\!\mathbb{R}\!\vec{w_i})\subset\mathbb{T}^{<\omega}\times\mathbb{R}^{<\omega}$$ where each $\vec{h}\in(\bigoplus_{i\in\beta(x)}\!\!\mathbb{T}\!\vec{v_i})\oplus(\bigoplus_{i\in\gamma(x)}\!\mathbb{R}\!\vec{w_i})$ represents the coset $\vec{h}+\,G_x$ and the operation on each $\mathbb{T}\!\vec{v_i}$ is given by the usual $\mathbb{T}$-operation $t\vec{v_i}+s\vec{v_i}=(t+s-\lfloor t+s\rfloor)\!\vec{v_i}$.\end{remark}

So we get the following theorem.
\begin{theorem}\label{reductiontofree}An equivalence relation induced by a Borel action $(g,x)\mapsto g\cdot x$ of $\mathbb{R}^{<\omega}$ on some standard Borel space $X$ is a countable disjoint union of equivalence relations each of which is induced by a free action of a countable sum of copies of $\mathbb{R}$ and $\mathbb{T}$.
\end{theorem}
\begin{proof} By Proposition \ref{effbasis} and the Remark, if we let $X_{\beta,\gamma}=\{x\in X:\beta(x)=\beta\,\,\&\,\,\gamma(x)=\gamma\}$ for each $\beta, \gamma\in\omega+1$, then the $X_{\beta,\gamma}$'s are disjoint invariant Borel sets where $x\in X_{\beta,\gamma}\Leftrightarrow \mathbb{R}^{<\omega}/G_{x}\cong(\bigoplus_{i\in\beta}\!\!\mathbb{T})\oplus(\bigoplus_{i\in\gamma}\!\mathbb{R})$. We may then define a free action $\cdot_{\beta,\gamma}$ of $(\bigoplus_{i\in\beta}\!\!\mathbb{T})\oplus(\bigoplus_{i\in\gamma}\!\mathbb{R})$ on $X_{\beta,\gamma}$ by $$(\vec{t},\vec{r})\cdot_{\beta,\gamma}x=\left(\sum_{i\in\beta}t_{i}\!\vec{v_i}(x)+\sum_{i\in\gamma}r_{i}\!\vec{w_i}(x)\right)\cdot x,$$ where if we let $E_{\beta,\gamma}=E^X_{\mathbb{R}^{<\omega}}\upharpoonright X_{\beta,\gamma}$, we see that for $x,y\in X_{\beta,\gamma}$, we have
\begin{align*}x E_{\beta,\gamma}y&\Leftrightarrow\exists \vec{g}\in\mathbb{R}^{<\omega}\,\, \text{such that}\,\, y=\vec{g}\cdot x\\
&\Leftrightarrow\exists\vec{h}\in(\bigoplus_{i\in\beta(x)}\!\!\mathbb{T}\!\vec{v_i})\oplus(\bigoplus_{i\in\gamma(x)}\!\mathbb{R}\!\vec{w_i})\,\, \text{such that}\,\, y\in(\vec{h}+G_x)\cdot x\\
&\Leftrightarrow\exists(\vec{t},\vec{r})\!\in\!(\bigoplus_{i\in\beta}\mathbb{T})\oplus(\bigoplus_{i\in\gamma}\mathbb{R})\, \text{such that}\, y=\!\left(\sum_{i\in\beta}t_{i}\!\vec{v_i}(x)+\sum_{i\in\gamma}r_{i}\!\vec{w_i}(x)\right)\!\cdot x
\end{align*}
So each $E_{\beta,\gamma}$ is induced by the free action $\cdot_{\beta,\gamma}$ of $(\bigoplus_{i\in\beta}\!\!\mathbb{T})\oplus(\bigoplus_{i\in\gamma}\!\mathbb{R})$, and $E^X_{\mathbb{R}^{<\omega}}=\bigsqcup_{\beta,\gamma}E_{\beta,\gamma}$.
\end{proof}

\section{Free actions of countable sums of $\mathbb{R}$ and $\mathbb{T}$}\label{freesection}
\subsection{Marker Sets for Locally Compact Group Actions}
We have seen from Theorem \ref{lacunary} that we can construct a Borel set which hits every orbit of an acting locally compact Polish group, and we can do so in a way that guarantees the points in the set which belong to the same orbit are ``spread out'' with respect to the acting group. However, this is not quite enough to give us the control we desire when trying to make local moves in our constructions. We want to be able to describe what goes on around a point using the nearest points of the spread out subset for reference. But we do not yet have a spread out subset for which we can guarantee one does not have to look too far in order to find a reference point. So now our goal is to show that we can extend Kechris' countable $K$-discrete sections into $K$-discrete sections with the added property that any point in the orbit is $K$-close to a point in the section.

While adding points to the section, we may along the way accidentally add points that are too close together. So then we want to be able to choose which points remain in the set and which ones to throw out in order to maintain the section's $K$-discreteness. To do this, we will make use of the following lemma from \cite{jackson_kechris_louveau}.

\begin{lemma}\label{maxdiscrete}
Let $F$ be a locally finite symmetric reflexive Borel relation on $X$. Then there exists a Borel maximal $F$-discrete subset of $X$.
\end{lemma}
Here, \emph{locally finite} means that $\forall x\in X$, $F(x)=\{y\in X:yFx\}$ is finite, and we say that $Y\subseteq X$ is \emph{maximal} \emph{F-discrete} if we have both $\forall x, y\in Y (x\neq y\implies\neg xFy)$ and $\forall x\in X \exists y\in Y (xFy)$. 

When applying Lemma \ref{maxdiscrete} in our argument, the relation $F$ that we are interested in is the $K$-relation, $xFy\iff y\in K\cdot x$, where $K$ is a compact symmetric neighborhood of the identity in the acting group. And to apply Lemma \ref{maxdiscrete}, we need to make sure each time that the $K$-relation is locally finite on its domain. The following lemma will make this much easier to manage.

\begin{lemma}\label{finite}
Suppose we have an action of a Hausdorff group $G$ on a set $X$. Let $K$ be a symmetric neighborhood of the identity in $G$, let $Y\subseteq X$ be $K$-discrete (i.e., if $x, y\in Y$ are distinct, then $y\nin K \cdot x$), and let $C\subseteq G$ be compact. Then for any $x\in X$, the set $Y\cap (C\cdot x)$ is finite.
\end{lemma}
\begin{proof}
Suppose for contradiction that $Y\cap (C\cdot x)$ is infinite. Then there exists an infinite subset $A\subseteq C$ where $A\cdot x\subseteq Y$ and if $g, h\in A$ are distinct then $ g\cdot x\neq h\cdot x$. But then since $A\subseteq C$ is infinite and $C$ is compact, $A$ must cluster at some $\hat{c}\in C$. Letting $\widehat{K}$ be a symmetric neighborhood of the identity such that $\widehat{K}^2\subseteq K$, and noting that the group is Hausdorff, it follows that we can choose distinct $a_1, a_2\in A\cap \widehat{K}\hat{c}$. But then $a_1\in\widehat{K}^{2}a_2\subseteq Ka_2$. So we would have $a_1\cdot x\in K\cdot(a_2\cdot x)$ where $a_1\cdot x$ and $a_2\cdot x$ are distinct members of $Y$, contradicting $Y$'s $K$-discreteness.
\end{proof}

Now, to extend Kechris' $K$-discrete sections into maximal ones, we will make a somewhat similar argument to one given by Slutsky in \cite{slutsky}. However, Slutsky's definition of $K$-lacunary is what we would call $K^2$-discrete, and the alternative proof provided here seems to be of value.

\begin{theorem}\label{Gmarker}
Let $G$ be a locally compact Polish group acting in a Borel manner on a standard Borel space $X$, and let $K$ be a compact symmetric neighborhood of the identity in $G$. Then there is a Borel set $M\subseteq X$ such that 
\begin{itemize}
\item[(i)] if $x, y\in M$ are distinct, then $y\nin K\cdot x$,
\item[(ii)] $K\cdot M=X$.
\end{itemize}\end{theorem}
\begin{proof}
Let $(g,x)\mapsto g\cdot x$ be the Borel $G$-action on $X$. By Theorem \ref{continuousaction}, there is a Polish topology on $X$ for which the action is continuous. So for the duration of the proof we may assume without loss of generality that $X$ has a fixed Polish topology and that the action of $G$ on $X$ is continuous. Then by Theorem \ref{lacunary} we may let $Y$ be a Borel subset of $X$ so that if $x, y\in Y$ are distinct then $y\nin K\cdot x$, and $G\cdot Y=X$. (i.e., $Y$ is $K$-discrete and $Y$ meets every orbit.) 

Now, let $\{d_{n}:n\in\mathbb{N}\}$ be a countable dense subset of $G$, let $Y_{0}=Y$, and let $\widetilde{Y}_0=Y_0\cup((d_{0}\cdot Y)\setminus (K\cdot Y_0))$. In other words, we get $\widetilde{Y}_0$ by adding on all the points of $d_{0}\cdot Y$ which are not already within $K$ of a point of $Y_0$. Now, $\widetilde{Y}_0$ may not be $K$-discrete anymore since, while none of the new points are within $K$ of the old points, some of the new points may be within $K$ of each other. However, if we consider the relation $F$ on $\widetilde{Y}_0$ where $xFy\iff y\in K\cdot x$, it follows that $F$ is symmetric and reflexive since $K$ is symmetric and contains the identity. And it follows from Lemma \ref{finite} that $F$ is locally finite since there can be only finitely many $y\in Y\cap (d_0^{-1}K\cdot x$) for which it's possible that we'd have $d_{0}\cdot y\in K\cdot x$. Thus, we may apply Lemma \ref{maxdiscrete} and let $Y_1$ be a Borel, maximal $F$-discrete (hence maximal $K$-discrete) subset of $\widetilde{Y}_0$. We then iterate this construction letting each $Y_{n+1}$ be a maximal $K$-discrete subset of $\widetilde{Y}_{n}=Y_n\cup((d_{n}\cdot Y)\setminus (K\cdot Y_n))$. Note that, since none of the new points of a $\widetilde{Y}_n$ are within $K$ of $Y_n$ and $Y_n$ is $K$-discrete, the maximal $K$-discrete subset $Y_{n+1}$ of $\widetilde{Y}_n$ contains all of $Y_n$ and so the $Y_n$'s are increasing.

Finally, let $M=\bigcup_{n}Y_{n}$. If $x, y\in M$ are distinct, then since the $Y_n$'s are increasing there is an $n$ large enough so that $x, y\in Y_n$ and hence $y\nin K\cdot x$. So $M$ is $K$-discrete and has property (i).

Now we suppose for contradiction that there exists an $x\in X$ such that $M\cap(K\cdot x)=\emptyset$ (i.e., $x\nin K\cdot M$). By the $K$-discreteness of $M$ and the compactness of $K^2$, Lemma \ref{finite} tells us that $M\cap(K^{2}\cdot x)$ must be finite, and so the set 
$$A=\bigcup_{y\in M\cap(K^{2}\cdot x)}(K\cdot y)$$
is closed in $X$ and does not contain $x$. Thus $X\setminus A$ is open and contains $x$. And then for any open symmetric neighborhood of the identity $O\subseteq K$ such that $O\cdot x\subseteq X\setminus A$, we have that $(O\cdot x)\cap (K\cdot M)=\emptyset$. But this could not happen since the $d_{n}$'s are dense and $O$ is open. Since $x\in G\cdot y_0$ for some $y_0\in Y$, we let $g\in G$ be such that $x=g\cdot y_0$, and note that $Og$ is a nonempty open subset of $G$. So since the $d_{n}$'s are dense it follows that $\exists\, l\in\mathbb{N}$ such that $d_{l}\in Og$. But then $d_{l}\cdot y_0\in Og\cdot y_0=O\cdot x$. Hence by the $l$\textsuperscript{\,th} stage of the construction there would have been a $d_{l}\cdot y_0\in d_{l}\cdot Y$ which belonged to $O\cdot x$ that would have been added to $\widetilde{Y}_{l}$. And so there must be some $y\in Y_{l+1}\subseteq M$ which is within $K$ of that $d_{l}\cdot y_0$ and hence within $K$ of $O\cdot x$, a contradiction. Thus every $x\in X$ is such that $(K\cdot x)\cap M\neq\emptyset$, and $M$ has property (ii). \end{proof}

\subsection{Marker Sets for Actions of $\mathbb{T}^{n}\times\mathbb{R}^{n}$}

While working with the free actions of each $G_{n}=\mathbb{T}^{n}\oplus\mathbb{R}^n$, we will fix the seminorm $||(\vec{t},\vec{r})||=\max\{|r_{i}| : 1\leq i\leq n\}$ and let $\rho:G_{n}\times G_{n} \rightarrow [0,\infty)$ be the induced pseudometric given by $\rho(g,h)=||g-h||$.

And then we let $\rho_{X}:X\times X\rightarrow [0,\infty)$ be the induced pseudometric on X given by
\begin{displaymath}
\rho_{X}(x,y)=\left\{
\begin{array}{lr}
\min\{||g||:g\cdot x=y\}, & \text{if}\; y\in G_{n}\cdot x,\\
\infty, & \text{if}\; y\nin G_{n}\cdot x.
\end{array}\right.
\end{displaymath}

We show that $\rho_{X}$ is indeed a well-defined pseudometric.
\begin{lemma} Suppose $(g,x)\mapsto g\cdot x$ is a Borel action of  $G_{n}=\mathbb{T}^{n}\oplus\mathbb{R}^n$ on a standard Borel space $X$. Then \begin{displaymath}
\rho_{X}(x,y)=\left\{
\begin{array}{lr}
\min\{||g||:g\cdot x=y\}, & \text{if}\; y\in G_{n}\cdot x,\\
\infty, & \text{if}\; y\nin G_{n}\cdot x.
\end{array}\right.
\end{displaymath}
is well-defined, and for any $x,y,z\in X$, $\rho_{X}(x,z)\leq\rho_{X}(x,y)+\rho_{X}(y,z)$.
\end{lemma}
\begin{proof}
By our Theorem \ref{continuousaction} from Becker and Kechris, we may fix a Polish topology on $X$ for which the action of $G_{n}$ is continuous. Then $A=\{g\in G_{n}:g\cdot x=y\}$ is closed in $G_n$ and attains its minimum seminorm since we may let $d$ be large enough so that the closed ball $\bar{B}_{\rho}(\vec{0},d)$ in $\mathbb{R}^{n}$ is such that $(\mathbb{T}^{n}\times\bar{B}_{\rho}(\vec{0},d))\cap A\neq\emptyset$, hence the minimum of $||g||$ over $A$ will occur within this $\mathbb{T}^{n}\times\bar{B}_{\rho}(\vec{0},d)$, and then the seminorm is a continuous function on the compact set $A\cap(\mathbb{T}^{n}\times\bar{B}_{\rho}(\vec{0},d)$). And for any $g,h\in G_n$ such that $g\cdot x=y$ and $h\cdot y=z$, it follows that $(g+h)\cdot x=z$ and of course $||g+h||\leq||g||+||h||$. Hence $\min\{||u||:u\cdot x=z\}\leq\min\{||g||:g\cdot x=y\}+\min\{||h||:h\cdot y=z\}$.
\end{proof}

We will omit the subscript X from $\rho_{X}$ whenever there's little chance of confusion. And we also define, for $A, B\subseteq X$ and $x\in X$, $\rho(x,A)=\inf\{\rho(x,y):y\in A\}$ and $\rho(A,B)=\inf\{\rho(x,y):x\in A\,\&\, y\in B\}$.

The backbone of our constructions is the existence of the following sets.

\begin{lemma}\label{basicmarker} Let $\mathbb{T}^{n}\times\mathbb{R}^{n}$ act in a Borel manner on a standard Borel space $X$, and let $d$ be a positive real number. Then there is a Borel set $M\subseteq X$ such that 
\begin{itemize}
\item[(i)] if $x, y\in M$ are distinct, then $\rho(x,y)>d$,
\item[(ii)] for any $x\in X$, $\rho(x,M)\leq d$.
\end{itemize}\end{lemma}
\begin{proof}
$\mathbb{T}^{n}\times[-d,d]^n$ is a compact symmetric neighborhood of the identity. So by Theorem \ref{Gmarker} there is a Borel set $M\subseteq X$ such that 
\begin{itemize}
\item if $x, y\in M$ are distinct, then $y\nin (\mathbb{T}^{n}\times[-d,d]^n)\cdot x$,
\item $(\mathbb{T}^{n}\times[-d,d]^n)\cdot M=X$.
\end{itemize}
Now, $y\nin \mathbb{T}^{n}\times[-d,d]^n\cdot x$ implies that no $g$ with $||g||\leq d$ can be such that $g\cdot x=y$. Hence $\rho(x,y)=\min\{||g||:g\cdot x=y\}>d$, and $M$ has property (i). And $\mathbb{T}^{n}\times[-d,d]^n\cdot M=X$ says that for any $x$ there is a $g\in\mathbb{T}^{n}\times[-d,d]^n$, hence $||g||\leq d$, where $g\cdot x\in M$. So then $\rho(x,M)\leq d$ and $M$ has property (ii).
\end{proof}

We refer to such a set $M$ as a \emph{marker set}, and we may say \emph{$d$-marker set} when we want to specify the distance.

\subsection{$\mathbb{T}^{n}$-invariant Rectangular Marker Regions}
Again we  fix  an $n\in\mathbb{N}$, a standard Borel space $X$, and a Borel action of $G=\mathbb{T}^{n}\oplus\mathbb{R}^{n}$ on $X$. However, we now assume that $G=\mathbb{T}^{n}\oplus\mathbb{R}^{n}$ acts freely on $X$. A subset $Y\subseteq X$ will be called \emph{$\mathbb{T}^{n}$-invariant} if $x\in Y\Rightarrow(\mathbb{T}^{n}\times\{0\}^{n})\cdot x\subseteq Y$.

By \emph{marker regions} we will simply mean the $R$-equivalence classes for some subequivalence relation $R$ of $E^X_{G}$. Since the action of $\mathbb{T}^{n}\oplus\mathbb{R}^{n}$ on $X$ is free, we may also treat each $E^X_{G}$-equivalence class as an affine copy of $\mathbb{T}^{n}\oplus\mathbb{R}^{n}$ using the correspondence $[x]=(\mathbb{T}^{n}\oplus\mathbb{R}^{n})\cdot x$. Then, for any subequivalence $R$ of $E^X_G$, equivalence class $[x]\in E^X_G$, and any $y\in[x]$, there is a unique $J\subset \mathbb{T}^{n}\oplus\mathbb{R}^{n}$ so that $[y]_R = J\cdot x$. We then call a marker region a \emph{$\mathbb{T}^{n}$-invariant half-open n-dimensional rectangle} if the corresponding $J$'s in $\mathbb{T}^{n}\oplus\mathbb{R}^{n}$ are of the form $\mathbb{T}^{n}\times\prod_{n}[a_n,b_n)$ where each $b_{n}-a_{n}>0$. (So, we are assuming the rectangles have faces perpendicular to the coordinate axes.) When we say \emph{edge lengths} of such a ``rectangle'', we mean the set of lengths of the intervals $[a_n,b_n)$ corresponding to any one of the $J$'s. Also, we say that a region is a \emph{$\mathbb{T}^n$-invariant half-open n-dimensional rectangular polyhedron} if it is a finite union of $\mathbb{T}^n$-invariant half-open n-dimensional rectangles. Then for each $i\leq n$ we define an \emph{i-face} of such a ``polyhedron'' $P$ to be a maximal $(n-1)$-dimensional subset $F$ of the boundary of $P$ such that if $x,y\in F \,\&\, (\vec{t},\vec{r})\cdot x=y$, then $r_{i}=0$. Note that this definition doesn't require our faces to be connected, but for our purposes this will not matter and actually helps in the constructions.

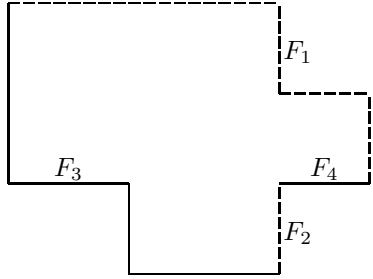
\begin{figure}[h]
\begin{center}
\setlength{\unitlength}{.2cm}
\begin{picture}(24,18)(0,0)
\put(0,6){\line(1,0){8}}
\put(8,0){\line(0,1){6}}
\put(0,6){\line(0,1){12}}
\put(8,0){\line(1,0){10}}
\multiput(18,0)(0,1){6}{\line(0,1){.75}}
\multiput(0,18)(1,0){18}{\line(1,0){.75}}
\put(18,6){\line(1,0){6}}
\multiput(18,12)(0,1){6}{\line(0,1){.75}}
\multiput(18,12)(1,0){6}{\line(1,0){.75}}
\multiput(24,6)(0,1){6}{\line(0,1){.75}}
\put(19.25,14){\makebox(0,0)[b]{$F_1$}}
\put(19.25,2){\makebox(0,0)[b]{$F_2$}}
\put(4,6.25){\makebox(0,0)[b]{$F_3$}}
\put(21,6.25){\makebox(0,0)[b]{$F_4$}}
\end{picture}\\

\small{In this example of a \emph{half-open rectangular polyhedron}, $F_1$ and $F_2$ are parts of the same \emph{face}. Similarly, $F_3$ and $F_4$ are parts of a single \emph{face}.}
\end{center}
\end{figure}

Now we will show that we may construct marker regions which are not only rectangular but are nearly square.

\begin{theorem}\label{squareregions} Let $d\geq\epsilon>0$. Then there is a Borel subequivalence relation $R_d$ of $E^X_G$ such that each $R_d$-marker region is a $\mathbb{T}^n$-invariant half-open n-dimensional rectangle with edges of lengths at least $d$ and less than $d+\epsilon$.
\end{theorem}
\begin{proof} Let $\Delta_2\gg\Delta_1\gg D\geq d\lceil\frac{d}{\epsilon}\rceil$. The result follows clearly from the following claims.

\textit{\underline{Claim 1}:  There is a Borel subequivalence relation $R_0$ of $E^X_G$ where each $R_0$-marker region is a $\mathbb{T}^n$-invariant half-open n-dimensional rectangular polyhedron where every pair of parallel faces have a perpendicular distance of at least $D$.}

\textit{\underline{Claim 2}:  If $R_0$ is a Borel subequivalence relation of $E^X_G$ where each $R_0$-marker region is a $\mathbb{T}^{n}$-invariant half-open n-dimensional rectangular polyhedron with parallel faces having a perpendicular distance of at least $D$, then there is a Borel subequivalence relation $R_D$ of $R_0$ so that every $R_0$-marker region is partitioned into $R_D$ marker regions which are $\mathbb{T}^{n}$-invariant half-open n-dimensional rectangles with edges of lengths at least D.}

\textit{\underline{Claim 3}:  If $R_D$ is a Borel subequivalence relation of $E^X_G$ where the $R_D$-marker regions are $\mathbb{T}^{n}$-invariant half-open n-dimensional rectangles with edge lengths at least $D$, then there is a Borel subequivalence $R_d$ of $R_D$ so that every $R_D$-marker region is partitioned into $R_d$ marker regions which are $\mathbb{T}^{n}$-invariant half-open n-dimensional rectangles with edges of lengths at least $d$ and less than $d+\epsilon$.}

\textit{Proof of Claim 1.}  We let $M$ be a basic Borel $\Delta_1$-marker set given by Lemma \ref{basicmarker}, and we let $K=\mathbb{T}^{n}\times[-\Delta_{2},\Delta_{2}]^n$.

Considering the relation $F$ on $M$ where $xFy\iff y\in K\cdot x$, it follows that $F$ is symmetric and reflexive since $K$ is symmetric and contains the identity. And it follows from Lemma \ref{finite} that $F$ is locally finite since $M$ is $(\mathbb{T}^{n}\times[-\Delta_{1},\Delta_{1}]^{n})$-discrete and $K$ is compact. Thus, we may apply Lemma \ref{maxdiscrete} and let $A_0$ be a Borel maximal $F$-discrete (hence maximal $K$-discrete) subset of $M$. We then iterate this construction letting $A_1$ be a maximal $K$-discrete subset of $M\setminus A_0$, letting $A_2$ be a maximal $K$-discrete subset of $M\setminus(A_{0}\cup A_{1})$, letting $A_3$ be a maximal $K$-discrete subset of $M\setminus(A_{0}\cup A_{1}\cup A_{2})$, and so on.

This defines a partition $A_{0},A_{1},\ldots, A_{k}$ of $M$ into disjoint Borel subsets where for any $i$ and $x\neq y\in A_{i}$, $\rho(x,y)>\Delta_{2}$. (The choice of \textit{maximal} $K$-discrete subsets guarantees that $k$ is finite.)

Now let $J=\mathbb{T}^{n}\times[-\Delta_{1}, \Delta_{1})^n$, and for each $x\in M$ let $R_{x}=J\cdot x$, so that $R_{x}$ is the $\mathbb{T}^{n}$-invariant half-open cubic region with ``center'' $\mathbb{T}^{n}\cdot x$ and edge lengths $2\Delta_{1}$. And we consider the relation $$x R_{M} y \Leftrightarrow \forall z\in M (x\in R_{z}\leftrightarrow y\in R_{z}).$$

Then by the definition of the pseudometric $\rho$ and the marker properties of $M$, the $R_{M}$ regions are already $\mathbb{T}^{n}$-invariant half-open n-dimensional rectangular polyhedra with faces perpendicular to the coordinate axes. But we need to modify these regions so that perpendicular distance between parallel faces is at least $D$. To do this, we define a collection $\{R'_{x} : x\in M\}$ of adjusted rectangles by inductively defining each $\{R'_{x} : x\in A_{i}\}$.

For $x\in A_{0}$ let $R'_{x}=R_{x}$. Then assuming that, for each $x, y\in \bigcup_{j<i}A_{j}$, we have 
\begin{itemize}
\item[(i)] $R_{x}\subseteq R'_{x}$,
\item[(ii)] the corresponding faces of $R_{x}$ and $R'_{x}$ have perpendicular distance no more than $\frac{1}{10}\Delta_{1}$, and
\item[(ii)] for each face $F_{1}$ of $R'_{x}$ and any parallel face $F_{2}$ of $R'_{y}$ with $\rho(x,y)\leq 3\Delta_{1}$, the perpendicular distance between $F_{1}$ and $F_{2}$ is at least $D$,
\end{itemize} we suppose $x\in A_i$ and let $R'_{y_{1}},\ldots, R'_{y_{m}}$ enumerate the surrounding rectangles where $y_{1},\ldots,y_{m}\in\bigcup_{j<i}A_{j}$ and $\rho(x,y_{l})\leq 3\Delta_{1}$ for each $1\leq l\leq m$. Note that Lemma \ref{finite} can be used again to verify $m$ must be finite, but in fact a volume argument would show that $m$ is bounded by $8^n$.

So then for each face of $R_{x}$ there are at most $2\cdot 8^n=2^{3n+1}$ many faces of $R'_{y_{1}},\ldots, R'_{y_{m}}$ parallel to it, and so if $\frac{1}{10}\Delta_{1}>2^{3n+1}(2D)$ then each face can be shifted away from the center $x$ so that the new faces satisfy (ii) and (iii). To define $R'_{x}$ for each $x\in A_{i}$ we simply shift each of $R_{x}$'s $2n$ faces away from the center in this way so that (i) will also hold for $\bigcup_{j<i+1}A_{j}$. Noting that for any $x,y\in A_{i}$, we have $\rho(x,y)>\Delta_{2}\gg\Delta_{1}$, we see that the constructions of any pair $R'_{x}$ and $R'_{y}$ at the $i^{\text{th}}$ stage do not affect each other. And it follows that the collection $\{R'_{x}: x\in M\}$ satisfies properties (i)-(iii) for all $x,y\in M$ and hence if we let $$x R_{0} y \Leftrightarrow \forall z\in M (x\in R'_{z}\leftrightarrow y\in R'_{z})$$ then $R_0$ is as desired.

\textit{Proof of Claim 2.} The proof of this claim requires no significant change from the argument in the $\mathbb{Z}^n$ case from \cite{gao_jackson}, but we provide it here anyway. To divide an $R_0$ region into the desired $R_D$ rectangles, we simply ``cut up'' each region by the linear expansions of its faces. For each $R_0$ region $P$, note that the faces of $P$ are
Borel subsets of $P$ which are definable from $P$.

Let $F_1,\dots, F_k$ be all the faces of $P$, and then for each $1\leq j\leq
k$, we let the face $F_j$ partition $P$ into at most two parts as follows. Say $F_{j}$ is an $i$-face. Then we define $F_j^+\subseteq P$ to be the set
of all $x\in P$ such that for any $y\in F_j$, letting $(\vec{t},\vec{r})$ be the
unique element of $\mathbb{T}^{n}\oplus\mathbb{R}^n$ with $(\vec{t},\vec{r})\cdot y=x$, the $i$-th coordinate of
$\vec{r}$ is non-negative. Let also $F_j^-=P-F_j^+$.  ($F_j^-$ could be
empty.)

Finally define a subequivalence relation $R_P$ on $P$ by
$$ xR_Py\iff \forall 1\leq j\leq k \,(\, x\in F_j^+\leftrightarrow
y\in F_j^+\,). $$ Then the equivalence classes of $R_P$ are $\mathbb{T}^{n}$-invariant half-open rectangles
whose faces are parts of linear expansions of the faces of $P$. These
rectangles have edge lengths greater than $D$ because the
parallel faces of $P$ have perpendicular distances greater than $D$.

\begin{figure}[h]
\begin{center}
\setlength{\unitlength}{.1cm}
\begin{picture}(60,55)(0,0)
\put(10,0){\line(0,1){20}}
\put(0,20){\line(1,0){10}}
\put(0,20){\line(0,1){15}}
\multiput(0,35)(1,0){22}{\line(1,0){0.5}}
\put(22,35){\line(0,1){20}}
\multiput(22,55)(1,0){23}{\line(1,0){0.5}}
\put(10,0){\line(1,0){25}}
\multiput(35,0)(0,1){10}{\line(0,1){0.5}}
\put(35,10){\line(1,0){25}}
\multiput(60,10)(0,1){33}{\line(0,1){0.5}}
\multiput(45,55)(0,-1){12}{\line(0,-1){0.5}}
\multiput(45,43)(1,0){15}{\line(1,0){0.5}}

\put(10,20){\line(0,1){15}}
\put(22,0){\line(0,1){35}}
\put(35,10){\line(0,1){45}}
\put(45,10){\line(0,1){33}}
\put(10,20){\line(1,0){50}}
\put(10,10){\line(1,0){25}}
\put(22,35){\line(1,0){38}}
\put(22,43){\line(1,0){23}}

\end{picture}\\
\small{Partitioning $P$ into $R_{P}$ regions by linear expansions of its faces.\\ (Points along  interior lines belong to the rectangle which is above and to the right.)}
\end{center}
\end{figure}

\textit{Proof of Claim 3.}  For each $x\in X$, let $\vec{l}(x)=(l_{1}(x),\ldots,l_{n}(x))$ where each $l_{i}(x)=\sup\{r\in\mathbb{R}:((\vec{0}, r e_{i})\cdot x)R_{D}x\} - \inf\{r\in\mathbb{R}:((\vec{0}, r e_{i})\cdot x)R_{D}x\}$. So $\vec{l}(x)$ measures the edge lengths of the sides of the $R_D$-marker region containing $x$, and $x R_D y \Rightarrow \vec{l}(x)=\vec{l}(y)$. We also measure the relative position (from below) of $x$ within its marker region by $\vec{p}(x)=(p_{1}(x),\ldots,p_{n}(x))$ where $p_{i}(x)=-(\inf\{r\in\mathbb{R}:((\vec{0}, r e_{i})\cdot x)R_{D}x\})$.

For each $1\leq i\leq n$, let $1\leq k_{i}(x)\in\mathbb{N}$ be least such that $l_{i}(x)-k_{i}(x)d<d$ and then let $\varphi_{i}(x)$ be the least $m\in\mathbb{N}$ be such that $$p_{i}(x)\leq(m+1)\!\left( d+\frac{l_{i}(x)-k_{i}(x)d}{k_{i}(x)}\right).$$

Now, define $$x R_d y \Leftrightarrow [x R_{D}y\,\,\text{and}\,\,\forall 1\leq i\leq n\,(\varphi_{i}(x)=\varphi_{i}(y))].$$

Then, noting that $l_{i}(x)-k_{i}(x)d<d\Rightarrow k_{i}(x)+1>l_{i}(x)/d>D/d>\lceil\frac{d}{\epsilon}\rceil$, hence $k_{i}(x)\geq\lceil\frac{d}{\epsilon}\rceil$, it follows that the $R_{d}$-marker regions have edge lengths at least $d$ and at most some $$d+\frac{l_{i}(x)-k_{i}(x)d}{k_{i}(x)}< d+\frac{d}{k_{i}(x)}\leq d+\frac{d}{\lceil\frac{d}{\epsilon}\rceil}\leq d+\epsilon.$$\end{proof}

\subsection{Orthogonal Marker Regions for Free Actions of $\mathbb{T}^{<\omega}\oplus\mathbb{R}^{<\omega}$}
Where $E$ is the orbit equivalence relation induced by an action of $\mathbb{T}^{<\omega}\oplus\mathbb{R}^{<\omega}$ we will use $E_n$ to denote the subequivalence relation induced by the action of the subgroup $\mathbb{T}^{n}\oplus\mathbb{R}^{n}$. (We will not return to using the usual meaning of $E_1$ until the end of the proof of Theorem \ref{infpowers}.)

We have established that, in the context of free actions of $\mathbb{T}^{n}\oplus\mathbb{R}^{n}$ on a standard Borel space $X$, we can build the appropriate marker sets, construct the partitions of the orbit equivalence classes into rectangular regions which are also $\mathbb{T}^{n}$-invariant, and also make the necessary adjustments to those regions in order to follow the Gao-Jackson machinery of \cite{gao_jackson}. And thus we may run essentially the same geometric construction of a sequence of orthogonal marker regions.

In particular, the proof of Lemma 5.2 of \cite{gao_jackson} depends on the same sort of adjustment of faces and subdivision of the polyhedral regions that we have illustrated in the proof of our Theorem \ref{squareregions}, and since the large scale geometry of $\mathbb{R}^n$ is the same as for $\mathbb{Z}^n$, only very superficial changes need be made in order to follow the same procedure for our Borel regions in $\mathbb{R}^{n}$ as for their clopen regions in $\mathbb{Z}^{n}$. (That we have really been working in $(\mathbb{R}^{n}\times\mathbb{T}^{n})/\mathbb{T}^n$ rather than $\mathbb{R}^{n}$ is a minor technicality as we do so by simply treating entire $\mathbb{T}^{n}$ orbits as if they are $\mathbb{R}^n$ points, which is not a problem since the $\mathbb{T}^{n}$'s are compact and so act smoothly and have Borel selectors. Thus, we may obtain the following analogue to Lemma 5.2 of \cite{gao_jackson}:

\begin{lemma} \label{morthogonal}
Let $R_1,\dots, R_k$ be a sequence of Borel subequivalence relations of
$E_{n}$ satisfying the following:
\begin{enumerate}[label=(\roman*)]
\item\label{morthogonal1}
On each $E_n$ class, each $R_i$ induces a partition into $\mathbb{T}^{n}$-invariant half-open polyhedral regions which are unions of $\mathbb{T}^{n}$-invariant half-open rectangles with edge lengths between $d$ and $12 d$.
\item
In any $\rho$-ball $B$ of radius $100,000 \cdot 16^n d$ in some $E_n$
equivalence class, there are at most $b$ integers $i$ such that one of the
$R_i$ regions has a face $\mathcal{F}$ which intersects $B$.
\end{enumerate}
Then there is a Borel subequivalence relation $\tilde{R}^n_d \subseteq E_n$
satisfying:
\begin{enumerate}[label=(\roman*)]
\item
Each $\tilde{R}^n_d$ class is a $\mathbb{T}^n$-invariant half-open rectangular region with edge lengths
between $9 d$ and $12 d$.
\item
Every face of an $\tilde{R}^n_d$ region is at least
$\frac{1}{9000^n 16^{n^2} b} d$ from any parallel face of an $R_i$ region (for any $i$).
\end{enumerate}
\end{lemma}

Here, $b$ should be thought of as a small integer to be determined later, and $d$ is sufficiently large that the fraction in property (ii) of the conclusion is $\geq 1$. Continuing to follow their procedure, we may construct a series of subequivalence relations with the analogous properties and get a similar inductive lemma as the one they arrive at for a free action of $\mathbb{Z}^{<\omega}$. What this allows us to do is begin with $R^i_i$'s that are the $i$-dimensional almost square regions from our Theorem \ref{squareregions} but with side lengths growing very fast. Then we refine the faces of each such region so that, restricted to any particular lower dimension, the projections of those ``fuzzy'' faces do not come too close to the faces of the previous almost square regions. This ``orthogonality'' limits the number of faces that are able to cut between any two equivalent points as we move up the dimensions, letting eventual agreement of the subequivalences to witness equivalence of points while also keeping a nice enough structure to the regions that we can definably pick points from each of them (namely their centers) along the way.

\begin{figure*}[h]
\begin{center}
\setlength{\unitlength}{.3cm}
\begin{picture}(40,13)(0,0)
\put(10,0){\makebox(0,0)[b]{$R^1_1$}}
\put(15,0){\makebox(0,0)[b]{$R^2_1$}}
\put(20,0){\makebox(0,0)[b]{$R^3_1$}}
\put(25,0){\makebox(0,0)[b]{$R^4_1$}}
\put(15,3.4){\makebox(0,0)[b]{$R^2_2$}}
\put(20,3.4){\makebox(0,0)[b]{$R^3_2$}}
\put(25,3.4){\makebox(0,0)[b]{$R^4_2$}}
\put(20,7){\makebox(0,0)[b]{$R^3_3$}}
\put(25,7){\makebox(0,0)[b]{$R^4_3$}}
\put(25,10.4){\makebox(0,0)[b]{$R^4_4$}}

\put(15,3.2){\vector(0,-1){1.6}}
\put(20,3.2){\vector(0,-1){1.6}}
\put(20,6.8){\vector(0,-1){1.6}}
\put(15,2.8){\vector(0,-1){1.2}}
\put(25,6.8){\vector(0,-1){1.6}}
\put(25,3.2){\vector(0,-1){1.6}}
\put(25,10.2){\vector(0,-1){1.6}}

\put(10.8,.6){\vector(1,1){3.4}}
\put(15.7,.7){\vector(1,2){3.5}}
\put(20.8,.8){\vector(1,3){3.4}}
\put(25.8,.8){\vector(1,4){3}}

\put(30,7){\makebox(0,0)[b]{$\vdots$}}
\put(30,8.5){\makebox(0,0)[b]{$\vdots$}}
\put(32,3.8){\makebox(0,0)[b]{$\cdots$}}
\put(34,3.8){\makebox(0,0)[b]{$\cdots$}}
\end{picture}
\end{center}
\end{figure*}

In particular, in the proof of the inductive Lemma 6.1 of \cite{gao_jackson}, we have the $R^{i}_{i}$ partitions that we constructed in our Theorem \ref{squareregions}. And when given the $R^{j+1}_{j},\ldots,R^{i-1}_j$ which for us are assumed to induce partitions into $\mathbb{T}^{n}$-invariant half-open polyhedral regions when restricted to $E_{j}$, then we may still apply our Lemma \ref{morthogonal} to produce a Borel $\mathbb{T}^n$-invariant rectangular partition $\tilde{R}$ which is orthogonal to them. And so then we may define $R^i_{j}$ in the same way by letting $$x R^i_{j} y \Leftrightarrow c([x]_{\tilde{R}}) R^i_{j+1} c([y]_{\tilde{R}}).$$ where we may still let $c(A)$ for an $\tilde{R}$ class $A$ be a center point of $A$ since $\mathbb{T}^n$ is compact and therefore $E^X_{\mathbb{T}^n}$ has a Borel selector. Hence we may define $c(A)$ by applying a Borel selector to the unique $E^X_{\mathbb{T}^{n}}$ class of points in $A$ which are equally $\rho_{X}$-distant from each pair of parallel faces of $A$. Note then that $R^{i}_{j}$ is also Borel. Thus, we have the following analogue to Lemma 6.1 of \cite{gao_jackson}.

\begin{lemma} \label{indstep}
Let $j<i$. Let $d_{j+1}$, $d_j$ be positive reals with $d_{j+1} \gg d_j$ (the exact condition
necessary is specified below).
Suppose $R^i_{j+1}$ is a Borel subequivalence relation of $E_i \subseteq E$.
Assume that the restriction of $R^i_{j+1}$ to each $E_{j+1}$ class induces a partition into $\mathbb{T}^{j+1}$-invariant polyhedral regions each of which is a
union of $\mathbb{T}^{j+1}$-invariant half-open rectangles with edge lengths between $d_{j+1}$ and $12 d_{j+1}$.
Suppose Borel equivalence relations $R^j_j$, $R^{j+1}_j, \dots$,
$R^{i-1}_j$ and $R^{j+1}_{j+1},\dots, R^{i-1}_{j+1}$ have been defined and satisfy:

\begin{enumerate}[label=(\roman*)]
\item
$R^j_j\subseteq E_j$; $R^k_j\subseteq E_k$ and $R^k_{j+1} \subseteq E_k$ for all $j < k \leq i-1$.
\item
$R^j_j$ induces a partition of each $E_j$ class into $\mathbb{T}^{j}$-invariant half-open rectangular regions
with edge lengths from $[d_j, d_j+\epsilon)$.
\item
For $j < k \leq i-1$, the restriction of $R^k_j$ to each $E_j$ class gives a partition
into $\mathbb{T}^j$-invariant polyhedral regions $R$ each of which is a union of $\mathbb{T}^{j}$-invariant half-open rectangles with edge lengths between $9 d_{j}$ and $12 d_j$.
\item \label{proph}
On each $E_j$ class, for each region $R$ induced by the restriction of $R^k_j$, there is a region $R'$ induced by the restriction of $R^k_{j+1}$  such that each face of $R$ is within $12 d_j$ of a face of $R'$.
\item \label{propb}
In any ball  $B$ of radius $100,000 \cdot 16^j d_j$ contained in an $E_j$ class,
there are at most $j+1$ many $k$ with $j < k \leq i-1$
such that some region induced by the restriction of $R^k_j$ has a face intersecting $B$.
\item  \label{propd}
For any $j < k_1 < k_2 \leq i-1$, and regions $R_1$, $R_2$ contained in an $E_{j}$
class induced by the restrictions of
$R^{k_1}_{j}$, $R^{k_2}_{j}$ respectively, if $\mathcal{F}_1$, $\mathcal{F}_2$ are parallel faces
of $R_1$, $R_2$, then $\rho(\mathcal{F}_1,\mathcal{F}_2) > \frac{1}{9000^{j} 16^{{j}^2} ({j}+1)} d_{j}$.
\item\label{propdj+1}
For any $j < k_1 < k_2 \leq i$, and regions $R_1$, $R_2$ contained in an $E_{j+1}$
class induced by the restrictions of
$R^{k_1}_{j+1}$, $R^{k_2}_{j+1}$ respectively, if $\mathcal{F}_1$, $\mathcal{F}_2$ are parallel faces
of $R_1$, $R_2$, then $\rho(\mathcal{F}_1,\mathcal{F}_2) > \frac{1}{9000^{j+1} 16^{(j+1)^2} (j+2)} d_{j+1}$.
\end{enumerate}

\noindent Then there is a Borel subequivalence relation $R^i_j \subseteq E_i$
satisfying the following:

\begin{enumerate}[label=(\roman*)]
\item
On each $E_j$ class, $R^i_j$ induces a partition into $\mathbb{T}^{j}$-invariant polyhedral regions $R$ each of which is
a union of $\mathbb{T}^{j}$-invariant half-open rectangles with edge lengths between $9 d_{j}$ and $12 d_j$.
\item \label{conc2}
On each $E_j$ class, for each region $R$ induced by $R^i_j$, there is a region $R'$ induced by
$R^i_{j+1}$ restricted to the $E_j$ class such that each face of $R$ is within $12 d_j$ of a face of $R'$.
\item
Condition \ref{propb} continues to hold, where now $j<k \leq i$.
\item\label{concd}
Condition \ref{propd} continues to hold,  where now $j<k_1 < k_2 \leq i$.
\end{enumerate}
\end{lemma}

We argue now that this construction still gives us what we need in our more general context of $\mathbb{T}^{<\omega}\oplus\mathbb{R}^{<\omega}$ acting on $X$. Namely, some relatively minor changes give us an analogue to Theorem 6.2 of \cite{gao_jackson}.

\subsection{Hypersmoothness of the Free Actions}

\begin{theorem}\label{infpowers}
Suppose $E$ is the orbit equivalence relation on a standard Borel space $X$ induced by a free Borel action of $\mathbb{T}^{<\omega}\oplus\mathbb{R}^{<\omega}$. Then $E$ is hypersmooth.
\end{theorem}
\begin{proof} 
Fix a sufficiently fast growing sequence
$\epsilon \ll 1 \ll d_1 \ll d_2 \ll \cdots$. And for each $i$, let $R^i_i$ be the subequivalence relation of $E_i$ as given by Theorem \ref{squareregions} taking $d=d_i$.
Inductively on $i$ we define the subequivalence relations
$R^i_i$, $R^i_{i-1}, \dots, R^i_1$. $R^i_i$ has already been defined,
and we assume that $R^i_{j+1}$ and all the $R^k_l$ for
$1 \leq l \leq k <i$ have been defined. In particular, all the
$R^j_j$, $R^{j+1}_j, \dots$, $R^{i-1}_j$ have been defined. Moreover, if $i>j+1$, all the $R^{j+1}_{j+1},\dots, R^{i-1}_{j+1}$ have also been defined. We assume
inductively that $R^i_{j+1}$ and the subequivalence relations $R^j_j,\dots,R^{i-1}_j$ and $R^{j+1}_{j+1},\dots, R^{i-1}_{j+1}$
satisfy the hypotheses of Lemma \ref{indstep}.
We then get $R^i_j$ from Lemma \ref{indstep} and are able to define $R^i_j$ so that the resulting sequences of subequivalence relations continue to satisfy these hypotheses.

Next we show that if $x, y \in X$ and $x E y$, then
for all large enough $i$ we have $x R^i_1 y$. To see this, suppose $x E y$, let $k_0$ be large enough so that $x E_{k_0} y$, and let $(\vec{t},\vec{r})\in \mathbb{T}^{k_0}\oplus\mathbb{R}^{k_0}$ be such that $y=(\vec{t},\vec{r})\cdot\, x$. Now, suppose for contradiction that $I=\{i\geq 1:\neg(x R^i_1 y)\}$ is infinite. Then, let the infinite sequence $(c_i)_{i\in I}\in [0,1]^\omega$ be such that each $(c_{i}\!\vec{t},c_{i}\!\vec{r})\cdot\, x$ is on the boundary of $x$'s $R^i_1$ region. Then, the sequence $(c_i)_{i\in I}$ must have at least one limit point $c\in[0,1]$, and we let $z=(c\!\vec{t},c\!\vec{r})\cdot\, x$. Note that still $(c\!\vec{t}, c\!\vec{r})\in\mathbb{T}^{k_0}\oplus\mathbb{R}^{k_0}$ and hence $z E_{k_0} x$. Letting $J\subseteq I$ be the set of all $i\in I$ where $z$'s distance from the boundary of $x$'s $R^i_1$ region is $<\epsilon$, we then fix $k>k_0$ and consider $k+1$ many elements of $J$ which are greater than $k$, say $i_1<\ldots<i_{k+1}$. Then, for each $1\leq m\leq k+1$, conclusion \ref{conc2} of Lemma \ref{indstep} can be applied iteratively to say that there must be a point $u_m$ on a boundary face $\mathcal{F}_m$ of a region induced by the restriction of $R^{i_m}_k$ to $z$'s $E_k$ class so that $$\rho(B(z,\epsilon),u_m)\leq 12(d_1 + d_2 + \cdots + d_{k-1}).$$ Hence any two of those faces must be be within $24(d_1 + d_2 + \cdots + d_{k-1})+2\epsilon$ of each other. But then since each of these $\mathcal{F}_m$'s are $(k-1)$-dimensional hyperplanes in a $k$-dimensional space and there are $k+1$ many of them, two of them must be parallel. It would follow from conclusion \ref{concd} of Lemma \ref{indstep} then that these parallel faces much be at least $$ \frac{1}{9000^{k} 16^{{k}^2} ({k}+1)} d_{k}$$ far apart. But since $\epsilon \ll d_1 \ll \cdots \ll d_{k-1} \ll d_k$, it must also be that $$ \frac{1}{9000^{k} 16^{{k}^2} ({k}+1)} d_{k}>24(d_1 + d_2 + \cdots + d_{k-1})+2\epsilon.$$ So we have a contradiction and $I=\{i\geq 1:\neg(x R^i_1 y)\}$ must be finite.

Now, we show that the construction of the $R^i_1$'s provides us a way to choose a point from each $R^i_1$ class. To do this, note that we may define the ``center'' of an $R^i_1$ class of $x$ to be the center of the rectangular $R^i_i$ class which is closest to its $R^i_1$ class. To be more precise, we define the ``center torus'' $C_{i}(x)$ to be $\mathbb{T}^{i}\cdot z$ where $z\in R^i_i$ is any point equally $\rho_{X}$-distant from the edges of the unique $R^i_i$ class which is within Hausdorff distance $12d_{1}+\cdots +12d_{i-1}$ to the $R^i_1$ class of $x$. (Note that this definition does not depend on the choice of $z$.) Then, since $\mathbb{T}^i$ is compact it follows that $E^X_{\mathbb{T}^i}$ has a Borel selector $s_{i}:X\rightarrow X$ such that $s_{i}(x)\in \mathbb{T}^{i}\cdot x$ and $s_{i}(x)=s_{i}(y)$ for all $y\in\mathbb{T}^{i}\cdot x$. And so we define $$\phi_{i}(x)=s_{i}(C_{i}(x)).$$

Since each $\phi_{i}(x)\in[x]_E$, it follows that if $\phi_{i}(x)=\phi_{i}(y)$ for any $i$ then $x E y$. Also, if $x E y$ then we have $x R^i_1 y$ and hence $\phi_{i}(x)=\phi_{i}(y)$ for all large enough $i$. Letting $\pi:X\rightarrow\mathbb{R}$ be a Borel bijection, it follows that $$f(x)=(\pi(\phi_{1}(x)),\pi(\phi_{2}(x)),\pi(\phi_{3}(x)),\ldots)$$ is a Borel reduction of $E$ to $E_1$.
\end{proof}

And finally we show that any sum of $\mathbb{R}$s and $\mathbb{T}$s must provide hypersmooth orbits with its free actions.
\begin{corollary}\label{freeactions} Suppose $E$ is the orbit equivalence relation on a standard Borel space $X$ which is induced by a free Borel action, $a:G\times X\rightarrow X$, of $$G=(\bigoplus_{\beta}\mathbb{T})\oplus(\bigoplus_{\gamma}\mathbb{R})$$ where $\beta,\gamma$ are each finite or $\omega$. Then $E$ is hypersmooth.
\end{corollary}
\begin{proof}
Let $H=(\bigoplus_{\beta'}\mathbb{T})\oplus(\bigoplus_{\gamma'}\mathbb{R})$ where $\beta'=0$ if $\beta=\omega$, $\beta'=\omega$ if $\beta$ is finite, $\gamma'=0$ if $\gamma=\omega$, and $\gamma'=\omega$ if $\gamma$ is finite. (Where $\lambda=0$, we identify $\bigoplus_{\lambda}A$ with the trivial group $\{0\}$.) Then, define the isomorphism $\pi:\mathbb{T}^{<\omega}\oplus\mathbb{R}^{<\omega}\rightarrow G\times H$ by 
$$\pi(\vec{t},\vec{r})=(\langle\vec{t}\restr\beta,\, \vec{r}\restr\gamma\rangle,\langle\vec{t}\!\!\,^*, \vec{r}\!\!\,^*\rangle)$$
where $\vec{t}\!\!\,^*=\vec{0}$ if $\beta=\omega$, $\vec{t}\!\!\,^*=(t_{\beta+n})_{n\in\omega}$ if $\beta$ is finite, $\vec{r}\!\!\,^*=\vec{0}$ if $\gamma=\omega$, and $\vec{r}\!\!\,^*=(r_{\gamma+n})_{n\in\omega}$ if $\gamma$ is finite.

Now, define the action of $G\times H$ on $X\times H$ by letting $(g,h)\cdot (x,h')=(a(g,x), h+h')$, and note that this is a free action. Then, $f:X\rightarrow X\times H$ defined by $f(x)=(x,e_H)$ is a reduction of $E$ to $E^{X\times H}_{G\times H}$ since $a(g,x)=y$ implies $(g,e_H)\cdot (x,e_H)=(y,e_H)$, and $(g,h)\cdot (x,e_H)=(a(g,x), h+e_H)=(y,e_H)$ implies $h=e_H$ and hence $\exists g\in G(a(g,x)=y)$. Since $G\times H$ is isomorphic to $\mathbb{T}^{<\omega}\oplus\mathbb{R}^{<\omega}$ and is acting freely on $X\times H$, $E^{X\times H}_{G\times H}$ is hypersmooth by Theorem \ref{infpowers}, and this shows that $E$ must also be hypersmooth. \end{proof}

\section{Proof of the Main Theorems} 
Finally, we have everything needed to prove the main theorems.

\subsection*{{Proof of Theorem \ref{sumtheorem}.}}\:\;
Suppose an equivalence relation E on a standard Borel space is induced by a Borel action of a group which is isomorphic to the sum of a countable abelian group $A$ with a countable sum of copies of $\mathbb{R}$ and $\mathbb{T}$. By Theorem \ref{reductiontoR}, $E$ is Borel reducible to an equivalence relation $E^X_{\mathbb{R}^{<\omega}}$ which is induced by a Borel action of $\mathbb{R}^{<\omega}$ on some standard Borel space $X$. Then by Theorem \ref{reductiontofree} this $E^X_{\mathbb{R}^{<\omega}}$ is a countable disjoint union of equivalence relations each of which is induced by a free action of a countable sum of copies of $\mathbb{R}$ and $\mathbb{T}$, and by Corollary \ref{freeactions} each of these is hypersmooth. By Lemma \ref{hypersmoothunion}, this means $E^X_{\mathbb{R}^{<\omega}}$ is hypersmooth. Hence $E$ is hypersmooth. \endprf

\subsection*{{Proof of Theorem \ref{LCAtheorem}.}}\:\;
Suppose an equivalence relation $E$ on a standard Borel space $X$ is induced by a Borel action of a second countable LCA group. Note that a second countable LCA group must be Polish (see Theorem 5.3 of \cite{kechris}, and recall that our definition of LCA assumes Hausdorff). By Theorem \ref{LCAuniversal}, $E$ is Borel reducible to an equivalence relation which is induced by a Borel action of $\mathbb{R}^{n}\times\mathbb{Z}^{<\omega}$, for some non-negative integer $n$, on a Borel subset $Y\subseteq X$, which is hypersmooth by Theorem \ref{sumtheorem}. Hence $E$ is hypersmooth. As an action of a locally compact Polish group, $E$ is also essentially countable by Corollary \ref{essentiallycountable}, and thus by Corollary \ref{hypersmoothtohyperfinite} it must be that $E$ is essentially hyperfinite.\endprf

\section{Recent Result on a Question of Hjorth, Ding and Gao}
Earlier drafts of this manuscript cited a question originally mentioned in passing by Hjorth in \cite{hjorth} and later given as a Conjecture by Ding and Gao in \cite{ding_gao}, a positive answer to which would have provided a significant improvement to Hjorth's\, $\ell^1$ dichotomy from \cite{hjorth}. The Conjecture was that any essentially countable Borel equivalence relation which is reducible to some orbit equivalence relation induced by a Borel action of an abelian Polish group must be essentially hyperfinite. The evidence for it seemed to grow as we have now seen positive answers to the countable abelian case in \cite{gao_jackson}, the non-archimedean abelian case in \cite{ding_gao}, and other cases here. However, this question was answered in the negative by the recent result of Allison in \cite{allison} that any treeable countable Borel equivalence relation is reducible to an orbit equivalence relation induced by a Borel action of an abelian Polish group, as a treeable countable Borel equivalence relation need not be hyperfinite.

\bibliographystyle{unsrt}

\bibliography{Cotton_LCA_Dec2021}

\end{document}